\documentclass[11pt]{amsart}

\usepackage[utf8]{inputenc}
\usepackage[T1]{fontenc}
\usepackage[french, english]{babel}

\usepackage{amsthm}
\usepackage{amsmath}
\usepackage{amssymb}
\usepackage{mathrsfs}
\usepackage{bbm}
\usepackage{graphicx}
\usepackage{listings}
\usepackage{mathrsfs}
\usepackage{enumerate}
\usepackage{pict2e}
\usepackage{stmaryrd}
\usepackage{mathtools}

\usepackage[all,cmtip]{xy}

\usepackage{multirow}

\usepackage{color}
\definecolor{marin}{rgb}   {0.,   0.1,   0.5} 
\definecolor{rouge}{rgb}   {0.8,   0.,   0.} 
\definecolor{sepia}{rgb}   {0.4,   0.25,   0.} 
\definecolor{mag}{rgb}   {0.3,   0,   0.3} 

\usepackage[colorlinks,citecolor=sepia,linkcolor=marin,urlcolor=mag ,bookmarksopen,bookmarksnumbered]{hyperref}

\newtheorem{theorem}{Theorem}[section]
\newtheorem{corollary}[theorem]{Corollary}
\newtheorem{lemma}[theorem]{Lemma}
\newtheorem{proposition}[theorem]{Proposition}
\newtheorem{definition}[theorem]{Definition}
\newtheorem{remark}[theorem]{Remark}

\begin{document}

\title[Dynamics of nonlinear Klein-Gordon equations on $\mathbb{S}^2$
]{Dynamics of nonlinear Klein-Gordon equations\\ in low regularity on $\mathbb{S}^2$}

\author{Joackim Bernier}

\address{\small{Laboratoire de Math\'ematiques Jean Leray, Universit\'e de Nantes, UMR CNRS 6629\\
2 rue de la Houssini\`ere \\
44322 Nantes Cedex 03, France
}}

\email{joackim.bernier@univ-nantes.fr}

\author{Beno\^it Gr\'ebert}

\address{\small{Laboratoire de Math\'ematiques Jean Leray, Universit\'e de Nantes, UMR CNRS 6629\\
2 rue de la Houssini\`ere \\
44322 Nantes Cedex 03, France
}}

\email{benoit.grebert@univ-nantes.fr}

\author{Gabriel Rivi\`ere}

\address{\small{Laboratoire de Math\'ematiques Jean Leray, Universit\'e de Nantes, UMR CNRS 6629\\
2 rue de la Houssini\`ere \\
44322 Nantes Cedex 03, France and Institut Universitaire de France, Paris, France}}

\email{gabriel.riviere@univ-nantes.fr}



\begin{abstract} We describe the long time behavior of small non-smooth solutions to the nonlinear Klein-Gordon equations on the sphere $\mathbb{S}^2$.
More precisely, we prove that the low harmonic energies (also called super-actions) are almost preserved for times of order $\varepsilon^{-r}$, where $r \gg 1$ is an arbitrarily large number and $\varepsilon \ll 1$ is the norm of the initial datum in the energy space $H^1\times L^2$. Roughly speaking, it means that, in order to exchange energy, modes have to oscillate at the same frequency. The proof relies on new multilinear estimates on Hamiltonian vector fields to put the system in Birkhoff normal form. They are derived from new probabilistic bounds on products of Laplace eigenfunctions that we obtain using Levy's concentration inequality.
\end{abstract} 
\maketitle

\setcounter{tocdepth}{1} 
\tableofcontents

\section{Introduction}

The linear Klein-Gordon equation classically appears as a natural first candidate to describe a relativistic version of quantum mechanics~\cite[Ch.~1]{BjDr64} and it can be written on the sphere as
$$
 \partial_t^2 \Phi(t,x) = \Delta \Phi(t,x) - \mu \Phi(t,x)
$$
where $\mu >0$ is an external parameter referred as the \emph{mass}\footnote{Physically speaking, $\mu$ is rather the square of the mass, up to taking $c=1$ and $\hbar=1$.}, $x\in \mathbb{S}^2$ (the Euclidean unit sphere of $\mathbb{R}^3$), $t\in \mathbb{R}$, $\Phi(t,x) \in \mathbb{R}$ and $\Delta$ denotes the Laplace--Beltrami operator on the sphere. As usual, we rewrite this evolution equation as a first order system
$$
\partial_t \begin{pmatrix} \Phi \\ \partial_t \Phi  \end{pmatrix} = \begin{pmatrix} 0& 1  \\
\Delta - \mu & 0 \end{pmatrix} \begin{pmatrix} \Phi \\ \partial_t \Phi  \end{pmatrix}
$$
and the change of variable 
\begin{equation}
\label{eq:defu}
u := (  \mu - \Delta)^{1/4} \Phi+ i (\mu- \Delta)^{-1/4} \partial_t\Phi
\end{equation}
makes the linear Klein--Gordon equation diagonal
$$
i \partial_t u = \sqrt{\mu - \Delta} \, u.
$$
Indeed, it is well known that the spherical harmonics (i.e. the restriction to $\mathbb{S}^2$ of homogeneous harmonic polynomials on $\mathbb{R}^3$) make the Laplace--Beltrami operator diagonal :
\begin{equation}
\label{eq:lapdiag}
L^2(\mathbb{S}^2 ; \mathbb{R}) = \overline{\bigoplus_{\ell \in \mathbb{N}} E_\ell } \quad \mathrm{where} \quad E_\ell =  \mathrm{Ker}(\Delta + \ell (\ell +1) \mathrm{Id}_{L^2})\simeq\mathbb{R}^{2\ell+1}
\end{equation}
is the space of spherical harmonics of degree $\ell$. In other words, the linear Klein--Gordon equation rewrites
$$
\forall \ell \in \mathbb{N}, \quad i \partial_t \Pi_\ell u = \omega_\ell \Pi_\ell u \quad \mathrm{where} \quad \omega_\ell := \sqrt{\ell (\ell+1) + \mu}
$$
and $\Pi_\ell $ denotes the orthogonal projector on $E_\ell$. 

On the one hand, it is relevant to note that the following quantities are constants of motion for the linear Klein--Gordon equation
$$
I_v(u(t)) = \Big| \int_{\mathbb{S}^2} u(t,x) v(x) \mathrm{d}\mathrm{vol}_{\mathbb{S}^2}(x) \Big|^2 \quad \mathrm{with} \quad \ell \in \mathbb{N}, v\in E_\ell .
$$ 
Actually, they describe accurately its dynamics (up to the exact values of the frequencies $\omega_\ell$). However, they are too sharp to survive to perturbations of the linear Klein--Gordon equation. Indeed, due to the multiplicities of the eigenvalues of the Laplace--Beltrami operator ($E_\ell$ is of dimension $2\ell+1$),  one could design spectral perturbations commuting with its vector field but destroying completely these constants of the motion (and so a fortiori we also expect the same phenomenon in the nonlinear case as in \cite{GV,GT12}). 

On the other hand, the harmonic energies (also called \emph{super-actions}) 
$$
J_{\ell}(u(t)) := \| \Pi_\ell u(t) \|_{L^2}^2 =: \mathcal{E}_\ell(\Phi(t),\partial_t \Phi(t)) 
$$
are much more robust constants of motion  because they do not describe the energy exchanges inside the clusters $E_\ell$. They only encode the energy preservation of each cluster. Note that they rewrite (in the original variables $(\Phi,\partial_t \Phi)$) as
\begin{equation}
\label{eq:def_harm_en}
\mathcal{E}_\ell(\Phi(t),\partial_t \Phi(t)) := (\ell (\ell+1) + \mu)^{1/2} \, \| \Pi_{\ell} \Phi(t) \|_{L^2}^2 + (\ell (\ell+1) + \mu)^{-1/2} \, \| \Pi_{\ell} \partial_t \Phi(t) \|_{L^2}^2.
\end{equation}

In this paper, we address the question of their preservation by a nonlinear perturbation of the linear Klein--Gordon equation. More precisely, we consider the nonlinear Klein-Gordon equation
\begin{equation}
\label{eq:KG}
\tag{KG}
 \partial_t^2 \Phi(t,x) = \Delta \Phi(t,x) - \mu \Phi(t,x) + g(x) (\Phi(t,x))^{p-1}
\end{equation}
where $p \geq 3$ is an integer and $g \in L^\infty(\mathbb{S}^2;\mathbb{R})$ is a given factor making the equation possibly inhomogeneous. The equation is naturally equipped with initial data $\Phi^{(0)} \in H^1(\mathbb{S}^2;\mathbb{R})$ and $ \dot{\Phi}^{(0)} \in  L^2(\mathbb{S}^2;\mathbb{R})$, i.e.
$$
\forall x\in \mathbb{S}^2,\quad \Phi(0,x) = \Phi^{(0)}(x) \quad \mathrm{and} \quad \partial_t \Phi(0,x) = \dot{\Phi}^{(0)}(x).
$$
Focusing only on small solutions, $\varepsilon := \| \Phi^{(0)}  \|_{H^1}+ \| \dot{\Phi}^{(0)}\|_{L^2} \ll 1$, \eqref{eq:KG} is a perturbation of the linear Klein--Gordon equation and the question of the preservation of the harmonic energies \eqref{eq:def_harm_en} makes sense. 

Since \eqref{eq:KG} is locally well-posed (see subsection \ref{sub:LWP} for details), the dynamics of \eqref{eq:KG} remain close to the dynamics of the linearized equation for times of order $\varepsilon^{-(p-2)}$. As a consequence, on such a time scale, the super-actions are almost preserved. However, their conservation on longer time scales is nontrivial. Actually, there exist counter-examples for similar systems : the cubic wave equation on $\mathbb{T}^2$  \cite{GGMP21} and the cubic Klein--Gordon equation on $\mathbb{S}^3$ with a unit mass \cite{BCEHLM17,CEL17}. Nevertheless, they are closely related to the existence of resonances (i.e. the frequencies $\omega_\ell$ have to be rationally linked) which only hold for exceptional values of the mass $\mu$.

For generic values of the mass $\mu$, in \cite{BDGS07} , Bambusi,  Delort, Gr\'ebert and Szeftel prove the almost preservation, for very long times, of the harmonic energies of the nonlinear Klein--Gordon equations on Zoll manifolds (which include $\mathbb{S}^d$ for all $d\geq 2$). Nevertheless, their result only hold for very smooth solutions (in particular $g$ has to be smooth). More precisely, they prove\footnote{Actually, they only prove a $\ell^\infty$ instead of the $\ell^1$ estimate \eqref{eq:BDGS} (see Remark 3.21 of \cite{BDGS07}). Indeed, since they are only really interested in the variations of the $H^s$ norm, they do not have written a sharp estimate on the variation of the super-actions. Nevertheless, estimate \eqref{eq:BDGS} would be a direct corollary of their proof. } that for all $r \gg 1$ chosen arbitrarily large, there exists $s_0(r) $ such that for all $s \geq s_0(r)$, provided that $\varepsilon$ (the norm of the initial datum $(\Phi^{(0)},\dot{\Phi}^{(0)})$ in $H^{s+1/2} \times H^{s-1/2}$) is small enough, while $|t|< \varepsilon^{-r}$, the solution to the nonlinear Klein--Gordon equation exists and it satisfies
\begin{equation}
\label{eq:BDGS}
|t|\leq \varepsilon^{-r} \quad \Rightarrow \quad \sum_{\ell \in \mathbb{N}} \langle  \ell \rangle^{2s} \big |   \mathcal{E}_\ell(\Phi(t),\partial_t \Phi(t)) - \mathcal{E}_\ell( \Phi^{(0)},\dot{\Phi}^{(0)}) \big| \lesssim \varepsilon^p.
\end{equation}
The main flaw of this result is the smoothness assumption $s \geq s_0(r)$. Indeed, in their construction, the smoothness parameter $s_0(r)$ grows at least linearly with respect to $r$. In other words, the longer the time during which they prove the preservation of the super-actions is, the smoother the solutions have to be. This smoothness assumption is crucial in their proof and is systematically used to prove similar results -- see e.g. \cite{Bou96,Bam03,BG06,CHL08a,Del12,Ime13,Del15,BFG20b}. Nevertheless, on simpler models, numerical experiments strongly suggest that this assumption is irrelevant (i.e. $s_0(r)$ should not depends on $r$), see e.g. \cite{CHL08a, CHL08b} for discussions about \eqref{eq:KG} on $\mathbb{T}$.

Actually, in \cite{BDGS07}, the authors are interested in the preservation of super-actions because they aim at proving the \emph{almost global well-posedness} of the equation (i.e. well-posedness for times of order $\varepsilon^{-r}$ with $r$ arbitrarily large). Roughly speaking, since
$$
\| u(t) \|_{H^s}^2 =  \sum_{\ell \in \mathbb{N}} \langle \ell \rangle^{2s} \mathcal{E}_\ell(\Phi(t),\partial_t \Phi(t)),
$$
they proceed by bootstrap : assuming that $\| u(t) \|_{H^s}^2 \leq 2 \| u(0)\|_{H^s}^2 \simeq \varepsilon^2$, they control the variations of the super-action using \eqref{eq:BDGS} and, as a corollary, they deduce the sharper estimate
$$
\| u(t) \|_{H^s}^2  =  \| u(0) \|_{H^s}^2  + \mathcal{O}(\| u(0) \|_{H^s}^p) .
$$

However, in low dimension ($d\leq 2$), it is well known that smoothness is not required to obtain solutions for very long times. Indeed, the preservation of the Hamiltonian  
\begin{equation}
\label{eq:hamKg}
\mathcal{H}(\Phi,\partial_t \Phi) =   \int_{\mathbb{S}^2} \ \frac{| \nabla \Phi(x) |^2}2 + \mu \frac{(\Phi(x))^2}2  + \frac{(\partial_t \Phi(x))^2}2 - \frac{g(x) (\Phi(x))^{p}}p   \,  \mathrm{d}\mathrm{vol}_{\mathbb{S}^2}(x)
\end{equation}
provides an a priori global control of the energy norm ($H^1\times L^2$) of small solutions (see Lemma \ref{lem:ell}). Hence, one can derive the global well-posedness of the Cauchy problem associated with \eqref{eq:KG} (provided that the initial data are small enough; see Proposition \ref{prop:GWP} for details). Therefore, it is all the more natural to try to remove the smoothness assumption $s\geq s_0(r)$ of  \cite{BDGS07} to control the variations of the harmonic energies.

In the following theorem, which is the main result of this paper, we control, without regularity assumption, the variations of the low super-actions:
\begin{theorem}\label{thm:main} For all $r\geq p$, all $\nu>0$ and almost all $\mu >0$, there exist $\varepsilon_0>0$, $C>0$ and $\alpha_r>0$ (depending only on $r$) such that, provided that $\varepsilon := \| \Phi^{(0)}  \|_{H^1}+ \| \dot{\Phi}^{(0)}\|_{L^2} < \varepsilon_0$, the global solution to \eqref{eq:KG} satisfies
$$
|t|<\varepsilon^{-r} \quad \Rightarrow \quad \forall \ell \in \mathbb{N}, \quad \big|\mathcal{E}_\ell( \Phi^{(0)},\dot{\Phi}^{(0)}) - \mathcal{E}_\ell(\Phi(t),\partial_t \Phi(t))\big| \leq C \langle \ell \rangle^{\alpha_r} \varepsilon^{p-\nu}.
$$
\end{theorem}

Let us compare this result with the one of \cite{BDGS07} (i.e. \eqref{eq:BDGS}). For low super-actions (i.e. $\ell \simeq 1$), Theorem \ref{thm:main} is much better as it provides the same control on the variations of the super-actions (up to the $\varepsilon^{-\nu}$ loss) without requiring any smoothness assumption. Conversely, contrary to \eqref{eq:BDGS}, due to the $\langle \ell \rangle^{\alpha_r}$ loss, our result does not provide any information about the variation of the very high super-actions (i.e. $\ell \gg \varepsilon^{-(p-2)/ \alpha_r}$). Nevertheless, since the loss with respect to $\ell$ is polynomial,  Theorem \ref{thm:main} provides a nontrivial control of the variations of some ``quite high'' super-actions (i.e.  $1 \ll \ell \ll \varepsilon^{-(p-2)/ \alpha_r}$). 

Using this optimization and the a priori control on the energy norm of the solutions, we derive the following corollary which can be viewed as a kind of weak orbital stability result.
\begin{corollary} \label{cor:main} For all $r\geq p$, $s<1/2$ and almost all $\mu >0$, there exist $\varepsilon_0>0$, $C>0$ and $\delta>0$ (which does not depend on $\mu$) such that, provided that $\varepsilon := \| \Phi^{(0)}  \|_{H^1}+ \| \dot{\Phi}^{(0)}\|_{L^2} < \varepsilon_0$, the global solution of \eqref{eq:KG} satisfies
$$
|t|< \varepsilon^{-r} \quad \Rightarrow \quad \| u(t) - \sum_{\ell \in \mathbb{N}} e^{-i H_\ell(t)} \Pi_\ell u(0)  \|_{H^s} \leq C \varepsilon^{1+\delta}
$$
where $H_\ell(t) : E_\ell \otimes \mathbb{C} \to E_\ell \otimes \mathbb{C}$ are Hermitian maps and $u\in C^0(\mathbb{R};H^{1/2})$ is defined by \eqref{eq:defu}.
\end{corollary}


\subsection*{Further bibliographical comments} The question of the stability of the linear dynamics makes sense for most nonlinear partial differential equations on confined domains. In high regularity, Birkhoff normal forms lead to many important successes in proving the stability of several other interesting systems : \cite{Bou96,BG06,GIP,FGL13,YZ14,BD17,BMP20,FI20,FI19} in the non-resonant case and \cite{Bam99,Bou00,KillBill,KtA} in the resonant case. 

For Klein--Gordon, the papers \cite{Bou96,Bam03,BG06,CHL08a,Del12,Ime13,Del15,BFG20b} provide results similar to the one of Bambusi, Delort, Gr\'ebert and Szeftel \cite{BDGS07} (i.e. preservation of the super-actions up to times of order $\varepsilon^{-r}$ with $r$ arbitrarily large) but hold on other manifolds or with quasi-linear perturbations. The works \cite{DS04,DS06,Del09,FZ10,DI17,FGI20} only reach shorter times of stability but improve the one given by the local well-posedness (i.e. they get stability for $|t|<\varepsilon^{-q}$ with $q>p-2$ but not arbitrarily large). On some manifolds, for high modes, due to the quasi-resonance (i.e. when the small divisors are too small), some of these time scales seem so far to be optimal.
We also mention the recent works \cite{GP16,BB20} about the existence of KAM tori for the nonlinear Klein--Gordon equations.

Very recently, in \cite{BG21}, the first two authors have introduced a new way of performing Birkhoff normal forms for Hamiltonians PDEs which, contrary to the previous results, allows to deal with non-smooth solutions. As in Theorem \ref{thm:main}, they prove almost-conservation, 
for very long times, in low regularity, of the low (super-)actions of several nonlinear dispersive PDEs on tori or boxes (including nonlinear Klein--Gordon equations on $[0,\pi]$ with homogeneous Dirichlet boundary conditions). Nevertheless, as discussed below, to be extended to more general domains (like spheres), this result require nontrivial multilinear vector field estimates. The derivation and the proof of these estimates on the sphere $\mathbb{S}^2$ are the main technical novelties of this paper (see Sections \ref{sec:basis} and \ref{sec:hamfo}).

\subsection*{Comments about the results} $\bullet$ The arbitrarily small loss $\varepsilon^{-\nu}$ in Theorem \ref{thm:main} is the same as the one of Theorem 1.21 in \cite{BG21} (about nonlinear Schr\"odinger equations on $\mathbb{T}^2$). It is due to the fact that, in dimension $2$, $H^1$ is not an algebra. 

\noindent $\bullet$ Reasoning as in Corollary 1.14 of \cite{BG21}, we could prove that Corollary \ref{cor:main} holds in the critical case $s=1/2$ provided that the initial data are a little bit smoother : $\varepsilon = \| \Phi^{(0)}  \|_{H^{1+\eta}}+ \| \dot{\Phi}^{(0)}\|_{H^{\eta}}$ for some $\eta >0$ (and $\delta$ would depend on $\eta$).

\noindent $\bullet$ We could consider much more general nonlinearity in \eqref{eq:KG} (e.g. nonlocal or nonpolynomial). Actually, we chose $g(x) (u(x))^{p-1}$ for simplicity.

\noindent $\bullet$ We are quite confident that our results could be extended to Zoll surfaces. Nevertheless, it would generate a lot a technicalities. It seems to us that we could adapt our multi-linear estimates by considering clusters of quasi-modes (as in \cite{BDGS07}) but the cohomological equations would be much harder to solve (because they would not be diagonal). Moreover, it would raise several interesting questions which deserve further investigations. For example, is it possible to prove the preservation of the low actions (i.e. not only the super actions) for very long times on a generic Zoll manifold and with a generic mass ? Somehow, it would be one way to prove the stability of the linear dynamics. 

\noindent $\bullet$ Conversely, it is not clear if a similar result could be proven in higher dimension (for example on $\mathbb{S}^3$). First, the equation would not be necessarily well-posed. Moreover, our method is strongly related to the fact that $H^1$ is an algebra (or almost an algebra like on $\mathbb{S}^2$).  Indeed, roughly speaking, the Birkhoff normal procedure generates vector fields of arbitrarily large order which are somehow similar to $(\Phi,\partial_t \Phi) \mapsto \Phi^n$ with $p\leq n \leq r+p$. Hence it looks unavoidable to require that the energy space is an algebra.

\subsection*{Comments about the proof}
The proof of our results follow the new Birkhoff normal form strategy introduced by the first two authors (see  \cite[\S1.4]{BG21} for an informal description of this new strategy). Roughly speaking, compared with \cite{Bam03,BG06}, it consists in removing terms which are usually small thanks to the smoothness assumption (and so which are unsolved in that case) using a stronger non-resonance condition. More precisely, we need that the small divisors are controlled by the smallest index instead of the third largest. Even if this new Diophantine condition may seem too restrictive, it is typically satisfied for \eqref{eq:KG} since the eigenvalues of $\sqrt{\mu-\Delta}$ accumulate polynomially fast on $\mathbb{Z}+1/2$, which is an affine lattice. Actually it is a quite  direct application of \cite[Prop.~2.1]{BG21} as explained in section \ref{sec:mass}.

Nevertheless, as usual, the implementation of a normal form procedure requires some structures on the nonlinear part of the vector field of the equation: it has to belong to a class of vector fields which is stable by Lie brackets, resolution of cohomological equations and whose vector fields enjoy good multi-linear estimates in the energy space (here $H^{1/2}$ with respect to the variable $u$ defined by \eqref{eq:defu}).  In \cite{BG21}, such classes have been developed to deal with Hamiltonian PDEs on tori (or boxes) in low regularity. Unfortunately, it seems hopeless to adapt them in more general domains like spheres as they strongly rely on the exceptionally good algebraic properties of the eigenfunctions of the Laplace operator (which are the complex exponentials). On spheres (and more generally on compact Riemannian manifolds), Delort and Szeftel have developed powerful classes of vector fields (see e.g. \cite{DS04, DS06}) on which most of the Birkhoff normal form results are based. Unfortunately, these classes  are unsuitable to work in low regularity as they require a lot of smoothness and it seemed unlikely to us that they could be adapted in low regularity. Hence, we chose to follow a slightly differerent route relying on probabilistic tools referred as Levy's concentration inequalities~\cite{Led01} (see Theorem \ref{thm:Levy}) in order to build the Hamiltonian classes adapted to our problem. See Section~\ref{sec:basis} for the probabilistic estimates and Section~\ref{sec:hamfo} for the multilinear vector field estimates.


\subsection*{Notations} 
It is natural (and usual) to index eigenvectors of the Laplace-Beltrami operators on $\mathbb{S}^2$ by points in a discrete triangle. As a consequence, for all $M \in (0,\infty]$, we define
$$
\mathcal{T}_M := \{ \ (\ell,m) \in \mathbb{N}\times \mathbb{Z} \quad | \quad 0\leq \ell \leq M \quad \mathrm{and} \quad -\ell \leq m \leq \ell  \ \}.
$$

We warn the reader that, as usual, we adopt the following convenient abuse of notation : being given $M>0$, $k \in \mathcal{T}_M$, $\sigma \in \{-1,1\}$ and $u=(u_{k'})_{k'\in\mathcal{T}_M} \in \mathbb{C}^{\mathcal{T}_M}$, we set  
$$
u_k^\sigma = u_k \quad \mathrm{if} \quad  \sigma=1 \quad  \mathrm{and} \quad u_k^\sigma = \overline{u_k} \quad \mathrm{if} \quad \sigma=-1.
$$

If $\mathfrak{p}$ is a parameter or a list of parameters and $x,y\in \mathbb{R}$ then we denote $x\lesssim_{\mathfrak{p}} y$ if there exists a constant $c(\mathfrak{p})$, depending continuously on $\mathfrak{p}$, such that $x\lesssim c(\mathfrak{p}) \, y$. Similarly, we denote $x\gtrsim_{\mathfrak{p}} y$ if $y\lesssim_{\mathfrak{p}} x$ and $x\approx_{\mathfrak{p}} y$ if $x\lesssim_{\mathfrak{p}} y\lesssim_{\mathfrak{p}} x$.

\subsection*{Acknowledgments} We thank Nicolas Burq for helpful discussions on global well-posedness. During the preparation of this work the authors benefited from the support of the Centre Henri Lebesgue ANR-11-LABX-0020-0 and by ANR-15-CE40-0001-01  "BEKAM". The third author was also supported by the Agence Nationale de la Recherche through the PRC grants ODA (ANR-18-CE40-0020) and ADYCT (ANR-20-CE40-0017). 

\section{A good orthonormal basis}


\label{sec:basis}
Recall that
\begin{equation}\label{eq:eigenspace}
E_\ell = \mathrm{Ker} (\Delta + \ell (\ell+1) \mathrm{Id}_{L^2(\mathbb{S}^2,\mathbb{R})})\simeq\mathbb{R}^{2\ell+1},
\end{equation}
and we will denote by $\mathcal{B}_\ell$ the set of orthonormal basis of the Euclidean space $E_\ell$. More generally, we denote by $\mathcal{B}$ the set of orthonormal basis of $L^2(\mathbb{S}^2;\mathbb{R})$:
$$
\mathcal{B}:=\left\{b=(b_\ell)_{\ell\in\mathbb{N}}:\ \forall\ell\geq 0,\ b_\ell\in\mathcal{B}_\ell\right\}.
$$
Hence, an element in $\mathcal{B}_\ell$ is an orthonormal basis of $E_\ell$ that we will denote by $b_\ell=(e_{\ell,m})_{-\ell\leq m\leq \ell}$ and an element of $\mathcal{B}$ can be represented as 
$$
b=(b_\ell)_{\ell\in \mathbb{N}}=(e_{k})_{k\in\mathcal{T}_\infty}=(e_{(\ell,m)})_{(\ell,m)\in\mathcal{T}_\infty}.
$$


When representing vector fields in a Hilbertian basis  $b=(e_k)_{k\in\mathcal{T}_\infty} \in \mathcal{B}$  (which seems natural to perform Birkhoff normal forms), it is classical to end up with estimating quantities of the following form
$$
\int_{\mathbb{S}^2} e_{k_1}(x)\ldots e_{k_p}(x) \mathrm{d}\text{vol}_{\mathbb{S}^2}(x),
$$
where $(k_1,\ldots, k_p)$ is some fixed element in $\mathcal{T}_\infty^p$. 
In the case of the round sphere, an orthonormal basis in $\mathcal{B}$ can be identified with a basis of homogeneous harmonic polynomials on $\mathbb{R}^3$ and one can make use of this structure to get good estimates. For instance, following~\cite[Ex.~4.2]{DS04}, we can verify that
\begin{equation}\label{eq:Delort-Szeftel}
\exists 1\leq j_0\leq r\ \text{such that}\ \sum_{j\neq j_0}\ell_j<\ell_{j_0}\ \Longrightarrow\ \int_{\mathbb{S}^2} e_{(\ell_1,m_1)}(x)\ldots e_{(\ell_p,m_p)}(x) \mathrm{d}\text{vol}_{\mathbb{S}^2}(x)=0.
\end{equation}
See also~\cite[Prop.~1.2.1]{DS06} for related results on more general manifolds. However, without any assumption on the relative size of the $\ell_j$, it seems that the best one can expect for a general orthonormal basis is to apply H\"older's inequality:
$$\left|\int_{\mathbb{S}^2} e_{(\ell_1,m_1)}(x)\ldots e_{(\ell_p,m_p)}(x) \mathrm{d}\text{vol}_{\mathbb{S}^2}(x)\right|\leq\|e_{(\ell_1,m_1)}\|_{L^p}\ldots\|e_{(\ell_p,m_p)}\|_{L^p}.$$
Then, a classical result on Laplace eigenfunctions~\cite{So88} states, for any $(\ell,m)\in\mathcal{T}_\infty$, $\|e_{(\ell,m)}\|_{L^p}\leq C_p\langle\ell\rangle^{\delta(p)}$ with $\delta(p)=\max\{\frac{1}{4}-\frac{1}{2p},\frac12-\frac2p\}$. Moreover, these bounds on $L^p$-norms are known to be sharp along certain sequences of the standard basis of spherical harmonics~\cite{So15}. Despite these a priori bounds and thanks to spectral degeneracies, there is some flexibility in the choice of the orthonormal basis $b\in\mathcal{B}$ we are working with. Following~\cite[Th.~6]{BL13} (see also~\cite{VdK97, ShZe03} or~\cite[Th.~18.5]{Ze08}), one can in fact prove that there exist many elements $b$ in $\mathcal{B}$ (in fact almost all) for which the $L^p$ norms are uniformly bounded. Thus, for such a basis $b$, one can find a constant $C_b>0$ such that, for every $(k_1,\ldots, k_p)\in\mathcal{T}_\infty^p$,
\begin{equation}\label{eq:Burq-Lebeau}
\left|\int_{\mathbb{S}^2}  e_{k_1}(x)\ldots e_{k_p}(x) \mathrm{d}\text{vol}_{\mathbb{S}^2}(x)\right|\leq C_b.
\end{equation}
Unfortunately, these informations do not seem to be enough to handle Birkhoff normal forms for data with low regularity as we are aiming at. Hence, we need to work a little bit more. As we shall see in the upcoming sections, the missing information to handle our Birkhoff normal form procedure is to construct an orthonormal basis in $\mathcal{B}$ for which these integrals have enough decay when there exists an index $1\leq j_0\leq p$ such that
$$(\ell_j,m_j)=(\ell_{j_0},m_{j_0})\ \Longrightarrow\ j=j_0.$$

To that aim, we will prove the following theorem which is the main result of this section:
 \begin{theorem}\label{t:proba}
Let $g\in L^\infty(\mathbb{S}^2;\mathbb{R})$ and let $p\geq 3$. Then, there exist a constant $C_{g,p}>0$ and an orthonormal basis $b=(e_k)_{k\in \mathcal{T}_\infty}\in C^\infty(\mathbb{S}^2;\mathbb{R})^{\mathcal{T}_\infty}$ of $L^2(\mathbb{S}^2; \mathbb{R})$ such that, for all $\mathbf{k} =(k_1,\ldots,k_p) \in \mathcal{T}_\infty^p$
we have
 \begin{equation}
 \label{eq:basis}
 \left| \int_{\mathbb{S}^2} e_{k_1}(x) \cdots e_{k_p}(x) \, g(x) \, \mathrm{d}\mathrm{vol}_{\mathbb{S}^2}(x) \right| \leq C_{g,p}\min\left\{1,  \frac{\log^{p}(2+|\ell|_{\infty})}{ \sqrt{\Upsilon(\mathbf{k})}  }\right\}.
 \end{equation}
 where $|\ell|_\infty = \max_{1\leq j\leq p} \ell_j$ and
 \begin{equation}\label{def:upsilon}
\Upsilon(\mathbf{k}) := \max\{ 1\} \cup \left\{\langle\ell_j\rangle:\forall j'\neq j,\ k_{j'}\neq k_{j}\right\}.
\end{equation}
 Moreover, $b\in\mathcal{B}$, i.e. for all $k=(\ell,m)\in \mathcal{T}_\infty$, we have
 $$
 \Delta e_{\ell,m} = - \ell (\ell+1) \, e_{\ell,m}.
 $$
 \end{theorem}
This theorem complements the properties given by~\eqref{eq:Delort-Szeftel} and~\eqref{eq:Burq-Lebeau} in the sense that it shows that the integrals of interest are small even if all the $\ell_j$ are of the same order. The only condition is that at least one of the eigenvector appears with multiplicity one in the integral. Note that the decay property we obtain is not that small but it will be enough for our argument. We do not expect that the decay can be much increased except in higher dimensions where the denominator should be $\langle\ell_j\rangle^{\frac{d-1}{2}}$ rather than $\langle\ell_j\rangle^{\frac{1}{2}}$. We emphasize that, contrary to~\eqref{eq:Delort-Szeftel}, this is not valid for any orthonormal basis but only for a generic one as~\eqref{eq:Burq-Lebeau} is. In order to prove this result, we will in fact refine the probabilistic approach used to prove~\eqref{eq:Burq-Lebeau}.

\begin{remark} As a corollary of the proof, we could also get a similar basis enjoying \eqref{eq:basis} for a countable set of degrees $p$ and functions $g$ (but not uniformly).
\end{remark}

\subsection{Probabilistic setup} We start with a short review on Haar measures which will be used to define natural probability measures on the orthogonal group of $E_\ell$. Then, we explain how to use these measures to define probability measures on $\mathcal{B}$ and how they are related to the normalized volume measure on the unit sphere $S_\ell$ of $E_\ell$.

\subsubsection{Background on Haar measures}

Recall that, given a compact group $G$, there exists a Radon measure $\mathfrak{m}_G$ on $G$ such that for every Borel subset $U\subset G$ and for every $g\in G$, $\mathfrak{m}_G(gU)=\mathfrak{m}_G(U)$~\cite[Th.~2.10]{Fo15}. This is called a (left invariant) \emph{Haar measure} on $G$ and for any nonempty open set $U$, one has $\mathfrak{m}_G(U)>0$~\cite[Prop.~2.19]{Fo15}. Moreover, if we fix $\mathfrak{m}_G(G)=1$, then this measure is unique~\cite[Th.~2.20]{Fo15}. The main example we will use in the following is the orthogonal group $O(d)$ of $\mathbb{R}^d$ (with $d\in\mathbb{N}^*$) or more generally, the orthogonal group $O(E)$ of some Euclidean space $E$ of dimension $d$. 
\begin{remark}\label{r:explicit}
 For the sake of concreteness, let us give an explicit expression of $\mathfrak{m}_{O(d)}$ in terms of measures on spheres. Given an orthonormal family $(X_1,\ldots, X_k)$ in $(\mathbb{R}^d)^k$, we denote by $\nu_{d-k-1}^{X_1,\ldots, X_k}$ the normalized volume measure on $\mathbb{S}^{d-1}\cap\operatorname{Span}\{X_1,\ldots,X_k\}^{\perp}$ induced by the Euclidean structure on $\mathbb{R}^{d-1}$. Equivalently,
$$\nu_{d-k-1}:=\frac{\operatorname{vol}_{\mathbb{S}^{d-1}\cap\operatorname{Span}\{X_1,\ldots,X_k\}^{\perp}}}{\operatorname{vol}_{\mathbb{S}^{d-k-1}}\left(\mathbb{S}^{d-k-1}\right)}.$$ 
With these conventions at hand and writing $R=(X_1,\ldots,X_d)\in O(d)$, one can verify using the invariance of $\nu_j$ by rotation that
$$\int_{O(d)}f(R) \mathrm{d}\mathfrak{m}_{O(d)}(R)=\int_{(\mathbb{S}^{d-1})^{d}}f(X_1,\ldots, X_d)\mathrm{d}\nu_0^{X_1,\ldots,X_{d-1}}(X_d)\ldots\mathrm{d}\nu_{d-2}^{X_1}(X_2)\mathrm{d}\nu_{d-1}(X_1).$$
In particular, if $f(R)=f(X_1,\ldots,  X_d)=g(X_1)$, then
$$\int_{O(d)}f(R) \mathrm{d}\mathfrak{m}_{O(d)}(R)=\int_{\mathbb{S}^{d-1}}g(X)\mathrm{d}\nu_{d-1}(X).$$
\end{remark}
If we now fix some compact subgroup $H$ of $G$, it also has a unique left invariant probability measure $\mathfrak{m}_H$. This measure is naturally related to $\mathfrak{m}_G$ as follows. We define $G/H:=\{[g]=gH: g\in G\}$ as the set of (left) cosets and according to~\cite[Th.~2.51, Cor.~2.53]{Fo15}, there exists some $G$-invariant measure $\mu_{G/H}$ such that, for every continuous function on $G$, one has
$$
\int_{G}f(g)\mathrm{d}\mathfrak{m}_{G}(g)=\int_{G/H} \left(\int_{H} f(gh)\mathrm{d}\mathfrak{m}_{H}(h) \right)\mathrm{d} \mu_{G/H}([g]),
$$
or more compactly
\begin{equation}\label{eq:subgroup} 
\mathfrak{m}_{G}=\int_{G/H} g_*(\mathfrak{m}_{H})\mathrm{d}\mu_{G/H}([g]).
\end{equation}
\begin{remark}
 Again, we will use this disintegration of the measure in the case of the orthogonal group $G=O(E)$ and of a subgroup $H=O(V)$, where $V$ is a linear subspace (with the same Euclidean structure) of $E$. Here an element $R\in O(V)$ is identified with an element of $O(E)$ by letting $R|_{V^\perp}=\operatorname{Id}_{V^\perp}$.
\end{remark}



\subsubsection{Probability measures on orthonormal basis}

The measure $\mathfrak{m}_{O(E_\ell)}$ induces a probability measure $\mathbb{P}_\ell$ on the set $\mathcal{B}_\ell$ of orthonormal basis of $E_\ell$ through the map 
$$R\in O(E_\ell)\mapsto \left(R \Phi_{(\ell,m)}\right)_{-\ell\leq m\leq\ell},$$
where $(\Phi_{(\ell,m)})_{-\ell\leq m\leq\ell}$ is a fixed orthonormal basis of $E_\ell$, e.g. the one given by the standard (real-valued) spherical harmonics. More generally, we define on the set $\mathcal{B}$ of orthonormal basis of Laplace eigenfunctions, the product measure 
$$\mathbb{P}=\bigotimes_{\ell=0}^{+\infty}\mathbb{P}_\ell.$$
If we fix some (nonempty) subset $L$ of $\mathbb{N}$, we can define the map
$$\pi_L:b=(b_{\ell})_{\ell\in \mathbb{N}}\in\mathcal{B}\mapsto (b_{\ell})_{\ell\in L}\in\mathcal{B}_L:=\prod_{\ell\in L}\mathcal{B}_\ell.$$
The pushforward $\mathbb{P}_L:=(\pi_L)_*\mathbb{P}$ is defined as
$$\int_{\mathcal{B}_L}f\mathrm{d}\mathbb{P}_L:=\int_{\mathcal{B}}f\circ\pi_L \mathrm{d}\mathbb{P},$$
and it can be written as
$$\mathbb{P}_L=\bigotimes_{\ell\in L}\mathbb{P}_\ell.$$
\begin{remark}
When $L=\{\ell\}$, we just write $\mathbb{P}_{\{\ell\}}=\mathbb{P}_\ell$ as we did so far. We will in fact mostly work with $\mathbb{P}_L$ for some finite set $L$. 
\end{remark}
We can also use the decomposition~\eqref{eq:subgroup} in that context. For instance, one can fix a subset $\mathcal{M}$ of $\{-\ell,\ldots,\ell-1,\ell\}$ and define 
$$V_{\ell,\mathcal{M}}:=\text{Span}\left\{\Phi_{\ell,m}:m\in \mathcal{M}\right\}.$$
Then, given an integrable function $f$ on $\mathcal{B}_\ell$, one can write
\begin{equation}\label{eq:desintegration}
 \int_{\mathcal{B}_\ell}f(b_\ell)\mathrm{d}\mathbb{P}_\ell(b_\ell)=\int_{O(E_{\ell})/ O(V_{\ell,\mathcal{M}})} \! \left(\int_{O(V_{\ell,\mathcal{M}})}\! \! \! f\left( (RR_1\Phi_{\ell,m})_{m}\right)\mathrm{d}\mathfrak{m}_{O(V_{\ell,\mathcal{M}})}(R_1)\right)\mathrm{d}\mu_{O(E_{\ell})/ O(V_{\ell,\mathcal{M}})}([R]).
\end{equation}
\begin{remark}
As $R_1\Phi_{\ell,m}=\Phi_{\ell,m}$ for $m\notin \mathcal{M}$ and for $R_1\in O(V_{\ell,\mathcal{M}})$, the integral 
$$\int_{O(V_{\ell,\mathcal{M}})}f\left( (RR_1\Phi_{\ell,m})_{m}\right)\mathrm{d}\mathfrak{m}_{O(V_{\ell,\mathcal{M}})}(R_1)$$
can be identified with an integral on the the set of orthonormal basis $\mathcal{B}_{\ell,\mathcal{M}}$ of $V_{\ell,\mathcal{M}}$ as we did above.
\end{remark}

\subsubsection{Induced measures on spheres} 

On the one hand, as we aim at finding an orthonormal basis $E_\ell$ with good properties via probabilistic means, it is natural to work with the Haar measure on the corresponding orthogonal group $O(E_\ell)$. On the other hand, our main probabilistic ingredient will be a result on the concentration of the volume measure on spheres of large dimensions as the unit sphere $S_\ell$ of $E_\ell$ is when $\ell\rightarrow+\infty$. As already witnessed from Remark~\ref{r:explicit}, the Haar measure is naturally related to such measures and, in view of our applications, we now make this connection slightly more precise in our context.

Fix $k=(\ell,m)$ in $\mathcal{T}_\infty$ and define the map 
$$\pi_{(\ell,m)}:b_\ell=(e_{(\ell,m')})_{-\ell\leq m'\leq \ell}\in \mathcal{B}_\ell\mapsto e_{(\ell,m)}\in S_\ell,$$ 
where $S_\ell$ is the unit sphere (for the $L^2$-norm) in $E_\ell$. The measure $\mathbb{P}_\ell$ induces a measure on the Euclidean sphere $S_\ell$ as follows:
\begin{equation}\label{eq:induced-measure-sphere}
\forall f\in\mathcal{C}^{0}(S_\ell),\quad\int_{S_\ell}f\mathrm{d}\nu_{2\ell}:=\int_{\mathcal{B}_\ell} f\circ \pi_{(\ell,m)}\mathrm{d}\mathbb{P}_\ell.
\end{equation}
By invariance of the Haar measure through orthogonal transformations, this measure does not depend on the choice of $m$. Still by definition of the Haar measure, one can also check that it is invariant under orthogonal transformations. Thus, by uniqueness of uniformly distributed measures on the sphere~\cite[Th.~3.4]{Ma95}, it can be identified with the \emph{normalized} volume measure $\nu_{2\ell}$ on the $2\ell$-dimensional sphere $S_\ell\simeq\mathbb{S}^{2\ell}$ of $E_\ell\simeq\mathbb{R}^{2\ell+1}$.
\begin{remark} In order to alleviate notations, rather than writing $\pi_{(\ell,m)}\circ\pi_\ell$, we shall also denote by $\pi_{(\ell,m)}$ the map from $\mathcal{B}$ to $S_\ell$ that associates to $b=(e_{(\ell',m')})_{(\ell',m')\in\mathcal{T}_{\infty}}$ the eigenvector $e_{(\ell,m)}$. The induced measure on $S_\ell$ remains the same by construction. 
\end{remark}

\subsection{The key probabilistic ingredient}

The key ingredient in the proof of~\eqref{eq:Burq-Lebeau} and of our proof of Theorem~\ref{t:proba} is the following property~\cite[Eq.~2.6]{Led01}
\begin{theorem}[Levy's inequality] \label{thm:Levy} Let $d\geq 1$ and let $\nu_{d}$ be the normalized volume measure on $\mathbb{S}^{d}$ induced by the Euclidean structure on $\mathbb{R}^{d+1}$. Let $F:\mathbb{S}^{d}\rightarrow \mathbb{R}$ be a continuous function. Then, for every $\delta>0$,
$$\nu_d\left(\left\{|F-m_F|\geq \omega_F(\delta)\right\}\right)\leq 2e^{-\delta^2\frac{d-1}{2}},$$
 where $m_F$ is the median of $F$, i.e. the unique real number such that
 $$\nu_d \left(\left\{F\geq m_F\right\}\right)=\nu_d\left(\left\{F\leq m_F\right\}\right)=\frac{1}{2},$$
 and where $\omega_F(\delta)$ is the modulus of continuity of $F$:
 $$\omega_{F}(\delta):=\sup\left\{|F(u)-F(v)|: d_{\mathbb{S}^{d}}(u,v)\leq\delta \right\}.$$
\end{theorem}
In other words, this theorem states that functions with small oscillations on spheres of large dimensions are almost constant. Following~\cite{ShZe03, Ze08, BL13}, let us illustrate how to use this theorem when $F_q(u):=\|u\|_{L^q(\mathbb{S}^2)}$ with $2\leq q<\infty$. Here $u$ belongs to $S_{\ell}$ that we identify with $\mathbb{S}^{2\ell}$ by fixing some orthonormal basis $(\Phi_{(\ell,m)})_{-\ell\leq m\leq \ell}$ of $E_\ell$. One has
\begin{equation*}
\begin{split}
|F_q(u)-F_q(v)|\leq \|u-v\|_{L^q(\mathbb{S}^2)}&\leq \|u-v\|_{L^2(\mathbb{S}^2)}^{\frac{2}{q}}\|u-v\|_{L^\infty(\mathbb{S}^2)}^{1-\frac{2}{q}} \\
&\leq \|u-v\|_{L^2(\mathbb{S}^2)}^{\frac{2}{q}}\Big(\sup_{x\in\mathbb{S}^2}\Big\vert\sum_{m=-\ell}^\ell \langle u-v,\Phi_{(\ell,m)}\rangle_{L^2} \Phi_{(\ell,m)}(x)\Big\vert\Big)^{1-\frac{2}{q}}\\
&\leq \|u-v\|_{L^2(\mathbb{S}^2)}\Big(\sup_{x\in\mathbb{S}^2}\Big\{\sum_{m=-\ell}^\ell \Phi_{(\ell,m)}(x)^2\Big\}\Big)^{\frac{1}{2}-\frac{1}{q}}.
\end{split}
\end{equation*}
Now observing that the sum is the Schwartz kernel of the spectral projector $\mathbf{1}_{\ell(\ell+1)}(-\Delta)$ evaluated on the diagonal and that this is a spherical invariant quantity, we deduce that these sums are independent of $x\in\mathbb{S}^2$ and thus equal to $2\ell+1$. Hence, there exists some constant $c_0>0$ such that, for every $\ell\geq 1$ and for every $2\leq q<\infty$,
$$|F_q(u)-F_q(v)|\leq \|u-v\|_{L^2(\mathbb{S}^2)}\left(2\ell+1\right)^{\frac{1}{2}-\frac{1}{q}}\leq c_0 d_{\mathbb{S}^{2\ell}}(u,v)\left(2\ell+1\right)^{\frac{1}{2}-\frac{1}{q}},$$
from which we infer the existence of $c_1>0$ (independent of $\ell$ and $q$) such that
$$
\forall\delta>0,\quad \nu_{2\ell}\left(\left\{u\in S_{\ell}:\ |\|u\|_{L^q}-m_{F_q}|\geq \delta\right\}\right)\leq 2e^{-c_1\delta^2\ell^{\frac{2}{q}}}.
$$
Finally, the constant $m_{F_q}$ can be estimated precisely through explicit calculations~\cite[Th.6]{BL13}. For our purpose, we shall only use the existence of a constant $c_2>\sqrt{2}$ such that, for every $2\leq q<\infty$, $1\leq m_{F_q}\leq c_2\sqrt{q}$~\cite[Th.4]{BL13}. In particular, there exists a constant $c_1>0$ such that, for every $\Lambda\geq 2c_2\sqrt{q}$, for every $\ell\geq 1$ and for every $2\leq q<\infty$, one has
\begin{equation}\label{eq:concentration-Lq}
\nu_{2\ell}\left(\left\{u\in S_{\ell}:\ \|u\|_{L^q}\geq \Lambda\right\}\right)\leq 2e^{-c_1(\Lambda-c_2\sqrt{q})^2\ell^{\frac{2}{q}}}.
\end{equation}

This quantitative estimate will be useful in our construction of a good orthonormal basis. Yet, besides these already known results, we will also need to apply Levy's inequality one more time directly to the integrals we are interested in. In order to clarify the upcoming argument, let us give another simple application of Levy's inequality that will be in the spirit of our proof. We fix some $h\in L^2(\mathbb{S}^2)$ and we consider the map
$$
F:u\in S_\ell\mapsto \int_{\mathbb{S}^2}u(x)h(x) \, \mathrm{d}\operatorname{vol}_{\mathbb{S}^2}(x)
$$
By symmetry, the median of this function is equal to $0$ and one has, thanks to the Cauchy--Schwarz inequality,
$$
|F(u)-F(v)|\leq \|u-v\|_{L^2} \|h\|_{L^2}\leq  c_0 \|h\|_{L^2} d_{\mathbb{S}^{2\ell}}(u,v).
$$
Hence, we deduce from Levy's inequality applied with $\delta=\frac{\log\langle \ell\rangle}{\sqrt{\langle\ell\rangle}}$ that
$$
\nu_{2\ell}\left(\left\{u\in S_{\ell}:\ |F(u)|\geq \frac{\log\langle \ell\rangle}{\sqrt{\langle\ell\rangle}}\right\}\right)\leq 2e^{-c_1\frac{\log^2\langle \ell\rangle}{\|h\|_{L^2}}}.
$$
From that, we infer that
$$
\sum_{k=(\ell,m)\in\mathcal{T}_\infty}\mathbb{P}\left(\left\{b\in\mathcal{B}: \left|\int_{\mathbb{S}^2}e_k(x)h(x) \, \mathrm{d}\operatorname{vol}_{\mathbb{S}^2}(x)\right|\geq \frac{\log\langle \ell\rangle}{\sqrt{\langle\ell\rangle}}\right\}\right)\leq 2\sum_{\ell\in\mathbb{N}}(2\ell+1)e^{-c_1\frac{\log^2\langle \ell\rangle}{\|h\|_{L^2}}}<\infty.
$$
In particular, thanks to the Borel Cantelli Lemma, we can derive that, given $h\in L^2$ and for $\mathbb{P}$-a.e. $b\in\mathcal{B}$, there exists a constant $C_b>0$ such that
$$
\forall k\in\mathcal{T}_\infty,\quad\left|\int_{\mathbb{S}^2}e_k(x)h(x) \, \mathrm{d}\operatorname{vol}_{\mathbb{S}^2}(x)\right|\leq C_b \frac{\log(1+\langle\ell\rangle)}{\sqrt{\langle\ell\rangle}}
$$
This is exactly the kind of decay we are looking for in Theorem~\ref{t:proba} except that $h$ is a product of eigenfunctions inside $b$ (rather than a fixed element $h$ in $L^2$). In order to handle this problem, we will make use of the fact that most eigenfunctions have their $L^q$-norm uniformly bounded and that this control on the $L^q$-norm can be made quantitative thanks to~\eqref{eq:concentration-Lq}. Due to the multiple and nested applications of Levy's inequality, this turns out to be a little bit tedious task. Yet, the decay phenomenon we obtain is the same as the one we have just described in this elementary calculation.

\subsection{Proof of Theorem~\ref{t:proba}} For the sake of simplicity, it is convenient to endow $\mathcal{T}_\infty$ with the lexicographic order, namely
\begin{equation}
\label{eq:order-relation}
k_1=(\ell_1,m_1)\preccurlyeq k_2=(\ell_2,m_2)\quad\Longleftrightarrow\quad \ell_1< \ell_2\ \text{or}\ (\ell_1=\ell_2\ \text{and}\ m_1\leq m_2).
\end{equation}

 We now will estimate the probability that an orthonormal basis in $\mathcal{B}$ does not satisfy the conclusion of Theorem~\ref{t:proba} for a fixed $\mathbf{k}=(k_1,\ldots, k_p)\in\mathcal{T}_\infty^p$ with
\begin{equation}\label{eq:ordered-index}
 k_1=(\ell_1,m_1)\preccurlyeq\ldots\preccurlyeq k_p=(\ell_p,m_p).
\end{equation}
Indeed, since the estimate of Theorem~\ref{t:proba} is invariant by the action of the permutation group on $\mathbf{k}$, we can can assume without loss of generality that $k_1,\ldots, k_p$ are ordered.

In order to alleviate the notations, we also define 
$$A(\mathbf{k}):=\{ k\in\mathcal{T}_{\infty}:\ \exists 1\leq j\leq p\ \text{such that}\ k=k_j\},$$
which is set of cardinal $\leq p$ so that
$$F_{\mathbf{k}}(b):=\int_{\mathbb{S}^2} e_{k_1}(x) \cdots e_{k_p}(x) \, g(x) \, \mathrm{d}\operatorname{vol}_{\mathbb{S}^2}(x)=\int_{\mathbb{S}^2} \prod_{k\in A(\mathbf{k})}e_{k}(x)^{\alpha_k} \, g(x) \, \mathrm{d}\operatorname{vol}_{\mathbb{S}^2}(x),$$
where $1\leq \alpha_k\leq p$ for every $k\in A(\mathbf{k}).$ We always suppose in the following that $g$ is not identically $0$.

\subsubsection{Applying Levy's inequality} 
We suppose that there exists $1\leq j_0\leq p$ such that
$$(\ell_j,m_j)=(\ell_{j_0},m_{j_0})\ \Longrightarrow\ j=j_0.$$
In that case, we say that $\mathbf{k}$ satisfies property $(S)$. We denote by $j_+$ the largest index in $\{1,\ldots,p\}$ with this property. In particular $\alpha_{(\ell_{j_+},m_{j_+})}=1$. We begin by treating the case of multi-indices verifying $(S)$ and we also suppose for the moment that $\ell_{j_+}\geq p$.

Following the above calculation, we aim at applying Levy's inequality to the map
$$F_+:e_{(\ell_{j_+},m_{j_+})}\in S_{\ell_{j_+}}\mapsto \int_{\mathbb{S}^2} e_{k_1}(x) \cdots e_{k_p}(x) \, g(x) \, \mathrm{d}\operatorname{vol}_{\mathbb{S}^2}(x),$$
with $(e_{k_j})_{1\leq j\neq j_+\leq p}$ fixed. By symmetry, the median $m_{F_+}$ of $F_+$ is equal to $0$. Moreover, by the H\"older inequality, this is a Lipschitz map:
\begin{eqnarray*}
|F_+(u)-F_+(v)| &\leq & \|g\|_{L^{\infty}}\|u-v\|_{L^2}  \left(\int_{\mathbb{S}^2}\prod_{j\neq j_+} |e_{k_j}(x)|^2 \mathrm{d}\operatorname{vol}_{\mathbb{S}^2}(x)\right)^{\frac{1}{2}}\\
&\leq & c_0\|g\|_{L^{\infty}}d_{\mathbb{S}^2}(u,v)  \left(\int_{\mathbb{S}^2}\prod_{j\neq j_+} |e_{k_j}(x)|^2 \mathrm{d}\operatorname{vol}_{\mathbb{S}^2}(x)\right)^{\frac{1}{2}}\\
&\leq & c_0\|g\|_{L^{\infty}}d_{\mathbb{S}^2}(u,v) \prod_{k\in A(\mathbf{k})\setminus \{k_{j_+}\}}\|e_{k}\|_{L^{2(p-1)}}^{\alpha_k}.
\end{eqnarray*}
\begin{remark}\label{r:compose-rotation}
 Note that these two properties also hold true for $F_+\circ R$ where $R\in O(E_{\ell_{j_+}}).$
\end{remark}

In order to apply Levy's inequality, we would at least need that the $L^{2p-2}$-norms appearing in the Lipschitz constant are uniformly bounded. To that aim, we set, for $\Lambda>0$,
$$B_{\Lambda}(\mathbf{k}):=\left\{b\in\mathcal{B}:\ \forall k\in A(\mathbf{k})\setminus \{k_{j_+}\}, \|\pi_{k}(b)\|_{L^{2(p-1)}}\leq\Lambda\right\}.$$
In particular, for $b\in B_{\Lambda}(\mathbf{k})$, the Lipschitz constant of $F_+$ is bounded by $c_0\|g\|_{L^{\infty}}\Lambda^{p-1}$. Moreover, using~\eqref{eq:concentration-Lq}, one finds that the complementary set of $B_{\Lambda}(\mathbf{k})$ is small. More precisely, for $\Lambda\geq 4c_2\sqrt{p}$, one has
\begin{eqnarray*}
\mathbb{P}\left(B_{\Lambda}(\mathbf{k})^c\right)&\leq&\sum_{k\in A(\mathbf{k})\setminus \{k_{j_+}\}}\mathbb{P}\left(\left\{b\in\mathcal{B}:  \|\pi_k(b)\|_{L^{2p-2}(\mathbb{S}^2)}\geq \Lambda\right\}\right)\\
&\leq &2\sum_{j=1,j\neq j_+}^pe^{-c_1(\Lambda-2c_2\sqrt{p})^2\ell_j^{\frac{1}{p-1}}}.
\end{eqnarray*}

Fix now some $\Lambda\geq 4c_2\sqrt{p}$ and some $\delta>0$. For $L\subset\mathbb{N}$, we set 
$$B_{\Lambda,L}(\mathbf{k}):=\left\{b\in\mathcal{B}_L:\ \forall k\in L\times\mathbb{Z}\cap\left( A(\mathbf{k})\setminus \{k_{j_+}\}\right), \|\pi_{k}(b)\|_{L^{2(p-1)}}\leq\Lambda\right\}$$
so that we can write
\begin{eqnarray*}
\mathbb{P}\left(\left\{b\in\mathcal{B}:\ |F_{\mathbf{k}}(b)|\geq \delta\right\}\right)&\leq & \mathbb{P}\left(\left\{b\in B_{\Lambda,\mathbb{N}}(\mathbf{k}):\ |F_{\mathbf{k}}(b)|\geq \delta\right\}\right)+2\sum_{j=1,j\neq j_+}^pe^{-c_1(\Lambda-2c_2\sqrt{p})^2\ell_j^{\frac{1}{p-1}}}\\
&\leq & \int_{  B_{\Lambda,\mathbb{N}\setminus\{\ell_{j_+}\}}(\mathbf{k}) }
\mathbb{P}_{\ell_{j_+}}\left(\left\{b_{\ell_{j_+}}\in B_{\Lambda,\ell_{j_+}}(\mathbf{k}): |F_{\mathbf{k}}(b',b_{\ell_{j_+}})| \geq\delta\right\}\right) \mathrm{d}\mathbb{P}_{\mathbb{N}\setminus\{\ell_{j_+}\}}(b')\\
&+&2\sum_{j=1,j\neq j_+}^pe^{-c_1(\Lambda-2c_2\sqrt{p})^2\ell_j^{\frac{1}{p-1}}}.
\end{eqnarray*}
Hence, $b'$ being fixed in $\mathcal{B}_{\mathbb{N}\setminus\{\ell_{j_+}\}}$, we are left with estimating, uniformly for $b'\in B_{\Lambda,\mathbb{N}\setminus\{\ell_{j_+}\}}(\mathbf{k})$,
\begin{equation}\label{eq:proba-smallsphere}
\mathbb{P}_{\ell_{j_+}}\left(\left\{b_{\ell_{j_+}}\in B_{\Lambda,\ell_{j_+}}(\mathbf{k}): |F_{\mathbf{k}}(b',b_{\ell_{j_+}})| \geq\delta\right\}\right),
\end{equation}
which can be analyzed using~\eqref{eq:desintegration}. Expressed in terms of the orthogonal group of $E_{\ell_{j_+}}$,~\eqref{eq:proba-smallsphere} can in fact be rewritten as
\begin{equation}\label{eq:proba-smallsphere-Haar}
\mathfrak{m}_{O(E_{\ell_{j_+}})}\left(\left\{R : \left(R\Phi_{\ell_{j_+},m}\right)_{m}\in B_{\Lambda,\ell_{j_+}}(\mathbf{k}),\ \text{and}\ \left|F_{\mathbf{k}}\left(b',\left(R\Phi_{\ell_{j_+},m}\right)_{m}\right)\right| \geq\delta\right\}\right).
\end{equation}
We are now exactly in position to apply the disintegration formula~\eqref{eq:desintegration} with $\ell=\ell_{j_+}$, $m_+=m_{j_+}$ and
$$
\mathcal{M}=\left\{k=(\ell,m)\notin A(\mathbf{k}): \ell=\ell_{j_+}\right\}\cup\{(\ell_{j_+},m_{j_+})\},
$$
where we note that $|\mathcal{M}|\geq 2(\ell_{j_+}+1)-p$. From this and as the condition on $B_{\Lambda,\ell_{j_+}}(\mathbf{k})$ only concerns indices $m$ not belonging to $\mathcal{M}$, we infer that~\eqref{eq:proba-smallsphere-Haar} (and thus~\eqref{eq:proba-smallsphere}) can be rewritten as
\begin{multline}
\int_{O(E_{\ell_{j_+}})/O(V_{\ell_{j_+},\mathcal{M}})}\mathbbm{1}_{\{[R]: (R\Phi_{\ell_{j_+},m})_{m}\in B_{\Lambda,\ell_{j_+}}(\mathbf{k})\}}([R])\\
\times\mathfrak{m}_{O(V_{\ell_{j_+},\mathcal{M}})}\left(\left\{R_1 :\left|F_{\mathbf{k}}\left(b',\left(RR_1\Phi_{\ell_{j_+},m}\right)_{m}\right)\right| \geq\delta\right\}\right)\mathrm{d}\mu_{O(E_{\ell_{j_+}})/O(V_{\ell_{j_+},M})}([R]).
\end{multline}
In order to estimate~\eqref{eq:proba-smallsphere} and thus $\mathbb{P}\left(\left\{b\in\mathcal{B}:\ |F_{\mathbf{k}}(b)|\geq \delta\right\}\right)$, we are left with determining an upper bound on 
$$\mathfrak{m}_{O(V_{\ell_{j_+},\mathcal{M}})}\left(\left\{R_1 :\left|F_{\mathbf{k}}\left(b',\left(RR_1\Phi_{\ell_{j_+},m}\right)_{m}\right)\right| \geq\delta\right\}\right),$$
uniformly for $b'\in B_{\Lambda,\mathbb{N}\setminus\{\ell_{j_+}\}}(\mathbf{k})$ and for $[R]$ such that $(R\Phi_{\ell_{j_+},m})_{m}\in B_{\Lambda,\ell_{j_+}}(\mathbf{k})$. Equivalently, as in~\eqref{eq:induced-measure-sphere}, one gets in terms of measures on spheres
$$\mathfrak{m}_{O(V_{\ell_{j_+},\mathcal{M}})}\left(\left\{R_1 :\left|F_{\mathbf{k}}\left(b',\left(RR_1\Phi_{\ell_{j_+},m}\right)_{m}\right)\right| \geq\delta\right\}\right)=\mathbb{\nu}_{|\mathcal{M}|-1}\left(\left\{u\in \mathbb{S}^{|\mathcal{M}|-1}: |F_{+}(Ru)| \geq\delta\right\}\right),$$
where $R$ is a fixed element in $E_{\ell_{j_+}}$ and where the function $F_+$ is defined using a fixed orthonormal family $\{e_{k_j}:1\leq j\neq j_+\leq p\}$ verifying $ \|e_{k_j}\|_{L^{2(p-1)}}\leq\Lambda$ for every $j\neq j_+$. Hence, using Levy's inequality and recalling from Remark~\ref{r:compose-rotation} that $F_+\circ R$ is Lipschitz and that its median is $0$, we obtain
$$\mathfrak{m}_{O(V_{\ell_{j_+},\mathcal{M}})}\left(\left\{R_1 :\left|F_{\mathbf{k}}\left(b',\left(RR_1\Phi_{\ell_{j_+},m}\right)_{m}\right)\right| \geq\delta\right\}\right)\leq 2e^{-\delta^2\frac{|\mathcal{M}|-2}{c_0^2\|g\|_{L^{\infty}}^2\Lambda^{2p-2}}}.$$
Gathering these bounds, we get
$$
\mathbb{P}\left(\left\{b\in\mathcal{B}:\ |F_{\mathbf{k}}(b)|\geq \delta\right\}\right) \leq 2e^{-\delta^2\frac{|\mathcal{M}|-2}{c_0^2\|g\|_{L^{\infty}}^2\Lambda^{2p-2}}}+2\sum_{j=1,j\neq j_+}^pe^{-c_1(\Lambda-2c_2\sqrt{p})^2\ell_j^{\frac{1}{p-1}}}.
$$
Note that, for $\ell_{j_+}\geq p$, one has $|\mathcal{M}|-2\geq 2\ell_{j_+}-p\geq \ell_{j_+}$.

In summary, we end up with the existence of two positive constants $c_1,c_2>0$ (depending only on $g$, on $p$ and on the geometry of $\mathbb{S}^2$) such that, for every $\delta>0$ and for every $\Lambda\geq 4c_2\sqrt{p}$ 
\begin{equation}\label{e:proof-proba-harder}
 \mathbb{P}\left(\left\{b\in\mathcal{B}:\ |F_{\mathbf{k}}(b)|\geq \delta\right\}\right) \leq 2e^{-c_1\frac{\delta^2 \ell_{j_+}}{\Lambda^{2p-2}}}+2\sum_{j=1,j\neq j_+}^pe^{-c_1(\Lambda-2c_2\sqrt{p})^2\ell_j^{\frac{1}{p-1}}},
\end{equation}
whenever $\mathbf{k}$ verifies $(S)$ and $\ell_{j_+}\geq p$. Taking $\Lambda=\log\langle\ell_p\rangle$ (and thus $\ell_p$ large enough), we can deduce the existence of a constant $c_{p,g}\geq 1$ such that, for every $\delta>0$ and for every $\mathbf{k}\in\mathcal{T}_\infty^p$ with $k_1\preccurlyeq\ldots\preccurlyeq k_p=(\ell_p,m_p)$ verifying $(S)$,
$$
 \mathbb{P}\left(\left\{b\in\mathcal{B}:\ |F_{\mathbf{k}}(b)|\geq \delta\right\}\right) \leq c_{p,g}e^{-c_{p,g}^{-1}\frac{\delta^2 \langle\ell_{j_+}\rangle}{\log^{2(p-1)}\langle\ell_p\rangle}}+c_{p,g}e^{-c_{p,g}^{-1}\log^{2}\langle\ell_p\rangle}.
$$
Thus, we obtain
\begin{equation}\label{e:proof-proba-harder-crude}
 \mathbb{P}\left(\left\{b\in\mathcal{B}:\ |F_{\mathbf{k}}(b)|\geq \frac{\log^{p}\langle\ell_p\rangle}{\sqrt{\langle\ell_{j_+}\rangle}}\right\}\right) \leq 2c_{p,g}e^{-c_{p,g}^{-1}\log^{2}\langle\ell_p\rangle},
\end{equation}

\subsubsection{The conclusion} Given $\mathbf{k}\in\mathcal{T}_\infty^p$ with $k_1\preccurlyeq\ldots\preccurlyeq k_p=(\ell_p,m_p)$ verifying property $(S)$ and $\ell_{j+}\geq p$, we define the following probabilistic events:
$$\Omega(\mathbf{k}):=\left\{b\in\mathcal{B}:\ |F_{\mathbf{k}}(b)|\geq \frac{\log^{p}\langle\ell_p\rangle}{\sqrt{\langle\ell_{j_+}\rangle}}\right\}.$$
Applying~\eqref{e:proof-proba-harder-crude}, one has
$$\sum_{k_1\preccurlyeq\ldots\preccurlyeq k_r: (S)\ \text{holds and $\ell_{j+}\geq p$}}\mathbb{P}(\Omega(\mathbf{k}))\leq C_{p,g}\sum_{\ell=1}^{+\infty}\ell^{2p} e^{-C_{p,g}^{-1}\log^2\langle \ell\rangle}<\infty.$$
In particular, thanks to the Borel-Cantelli Lemma, we can conclude that, for $\mathbb{P}$-a.e. $b\in\mathcal{B}$, one has $b\in\Omega(\mathbf{k})^c$ except for finitely many $\mathbf{k}$ verifying $(S)$ and $\ell_{j+}\geq p$. This yields the conclusion of the Theorem for indices verifying these two properties. Recall now from\footnote{This is in fact a rather direct consequence of~\eqref{eq:concentration-Lq} combined with H\"older's inequality and the Borel-Cantelli Lemma.}~\cite[Th.6]{BL13} that, for $\mathbb{P}$-a.e. $b\in\mathcal{B}$, there exists a constant $C_b>0$ such that, for every $\mathbf{k}=(k_1,\ldots,k_p)\in\mathcal{T}_\infty^p$,
$$
\left| \int_{\mathbb{S}^2} e_{k_1}(x) \cdots e_{k_p}(x) \, g(x) \, \mathrm{d}\mathrm{vol}_{\mathbb{S}^2}(x) \right| \leq C_{b}.
$$
This last inequality yields the conclusion of the theorem whenever $\mathbf{k}$ does not satisfy $(S)$ or $\ell_{j+}\geq p$. Hence, taking an element in the intersection of these two subsets of full measure concludes the proof of Theorem~\ref{t:proba}.


\begin{remark}
 We note that we proved something slightly stronger than what was stated in Theorem~\ref{t:proba} as the conclusion holds true for $\mathbb{P}$-a.e. orthonormal basis in $\mathcal{B}$ (with a constant that depends on the choice of $b$).
\end{remark}

\section{A good mass}
\label{sec:mass}
 In this section, we prove that, for almost all mass $\mu >0$, the frequencies of \eqref{eq:KG} are non-resonant and thus well-suited to proceed to a Birkhoff normal form reduction. The frequencies of \eqref{eq:KG} are defined by
\begin{equation}
\label{eq:def_omega}
\forall k=(\ell,m) \in \mathcal{T}_\infty, \quad \omega_k := \sqrt{\ell(\ell+1) + \mu}. 
\end{equation}
They are the eigenvalues of the operator $\sqrt{\mu - \Delta}$ (see \eqref{eq:lapdiag}).


The Birkhoff normal form process involves small divisors of the form 
\begin{equation}
\label{eq:defOmega}
\Omega(\sigma, \mathbf{k}) = \sigma_1 \omega_{k_1} + \cdots + \sigma_r \omega_{k_r}
\end{equation}
 with $r\geq 3$, $\sigma \in \{-1,1\}^r$ and $\mathbf{k} \in \mathcal{T}_\infty^r$. Of course there may be cancellations in these small divisors (a same term could appear both with a sign plus and a sign minus). Therefore it is useful to define the \emph{smallest effective index} by
 \begin{equation}
 \label{eq:defkappa}
 \kappa(\sigma,\mathbf{k}) = \min \big\{ \ \langle \ell_j \rangle \quad | \quad 1\leq  j \leq r \quad \mathrm{and} \quad \ \sum_{\ell_i = \ell_j} \sigma_i \neq 0 \ \big\} \cup \{+\infty\}.
 \end{equation}
 where, for all $i\in \llbracket 1,r \rrbracket$, we have set $(\ell_i,m_i) := k_i$. The following proposition provides a quite uniform lower bound for the small divisors of \eqref{eq:KG}.
\begin{proposition} \label{prop:strong_nr} For almost all $\mu>0$ and all $r\geq 2$, there exist $\gamma_r, \alpha_r>0$ such that for all $\mathbf{k}\in \mathcal{T}_\infty^r$, all $\sigma \in \{-1,1\}^r$, we have either
\begin{equation}
\label{eq:strong_nr}
|\Omega(\sigma, \mathbf{k})|\geq \gamma_r \kappa(\sigma,\mathbf{k})^{-\alpha_r} 
\end{equation}
or $\kappa(\sigma,\mathbf{k}) = +\infty$, i.e. $r$ is even and there exists $\rho$ in the symmetric group $ \mathfrak{S}_{r}$ such that
$$
\forall j \in \llbracket 1,r/2 \rrbracket, \quad \sigma_{\rho_{2j-1}}= -\sigma_{\rho_{2j}} \quad \mathrm{and} \quad \omega_{k_{\rho_{2j-1}}}=\omega_{k_{\rho_{2j}}}.
$$ 
Moreover, $\alpha_r$ does not depends on $\mu$.
\end{proposition}

As already explained in the introduction, the key observation here is that the small divisors that will appear in our normal formal procedure (see the proof of Theorem~\ref{thm_NF}) are controlled by the smallest effective index rather than the third largest index as for instance in~\cite[Prop.~3.16]{BDGS07}. This will allow us to remove much more terms when solving cohomological equations.
\begin{proof} First we note that the frequencies accumulates polynomially fast on lattice $\mathbb{Z}+\frac12$ :
$$
\omega_{(\ell,m)} = \sqrt{\ell (\ell+1) + \mu } = \ell \sqrt{1+ \frac1\ell + \frac{\mu}{\ell^{2}}} \mathop{=}_{\ell \to + \infty} \ell + \frac12 + \mathcal{O}\big(\frac1\ell\big).
$$
Moreover, it is well known (see e.g.~\cite[Prop.~4.8]{DS04} and \cite[Th.~6.5]{Bam03}) that Proposition \ref{prop:strong_nr} holds if \eqref{eq:strong_nr} is replaced by the weaker estimate
$$
\forall y\in \mathbb{Z}, \quad \big|\frac{y}2+\Omega(\sigma, \mathbf{k})\big|\geq \gamma_r \, \big(\max_{j=1}^r \langle k_j \rangle\big)^{-\alpha_r} 
$$
Therefore, Proposition \ref{prop:strong_nr} is a consequence of \cite[Prop.~2.1, p.11]{BG21} which only requires the two above ingredients.
\end{proof}

\section{Hamiltonian formalism}  
\label{sec:hamfo}

We now introduce new families of norms on real-valued and homogeneous polynomials on $\mathbb{C}^{\mathcal{T}_M}$ that are well behaved with respect to the canonical symplectic structure on $\mathbb{C}^{\mathcal{T}_M}$ and thus well adapted to our initial PDE problem after diagonalization of $\Delta$.

\subsection{Functional setting} We use the standard functional setting to deal with Hamiltonian systems. Nevertheless to avoid any possible confusion we recall it precisely (and we refer to section 3.1 of \cite{BG21} for more comments and details).

We consider $M\in (0,\infty)$ as a fixed parameter and we note that $\mathbb{C}^{\mathcal{T}_M}$ is a real finite dimensional vector space. We always consider this space as an Euclidean space for the $\ell^2$ scalar product
$$
\forall u,v \in \mathbb{C}^{\mathcal{T}_M}, \quad (u,v)_{\ell^2} := \Re \sum_{k\in \mathcal{T}_M} u_k \overline{v_k}.
$$ 
As a consequence, if $H: \mathbb{C}^{\mathcal{T}_M} \to \mathbb{R}$, we have the relation
$$
\forall k\in \mathcal{T}_M, \quad \frac{(\nabla H)_k}2 = \partial_{\overline{u_k}} H=: \frac{1}{2}\left(\partial_{\Re u_k} H + i  \partial_{\Im u_k} H\right).
$$
As usual, we equip implicitly $\mathbb{C}^{\mathcal{T}_M}$ with the symplectic form $( i \, \cdot \, , \, \cdot \,)_{\ell^2}$. Therefore a smooth map $\tau : \mathcal{D} \to \mathbb{C}^{\mathcal{T}_M}$, where $\mathcal{D}$ is an open set of $\mathbb{C}^{\mathcal{T}_M}$, is \emph{symplectic} if
$$
\forall u \in \mathcal{D}, \forall v,w \in \mathbb{C}^{\mathcal{T}_M}, \ (iv,w)_{\ell^2} = (i\mathrm{d}\tau(u)(v),\mathrm{d}\tau(u)(w))_{\ell^2}.
$$
Moreover, if $H,K :\mathbb{C}^{\mathcal{T}_M} \to \mathbb{R}$ are two smooth functions, the \emph{Poisson bracket} of $H$ and $K$ is defined by
$$
\{  H,K\}(u):= (i \nabla H(u),\nabla K(u))_{\ell^2}.
$$
Note that, as usual, it can be checked that we have 
$$
\{  H,K\}    =\sum_{k\in \mathcal{T}_M} \partial_{\Re u_k}H \partial_{\Im u_k} K - \partial_{\Im u_k}H \partial_{\Re u_k} K=2i \sum_{k\in \mathcal{T}_M} \partial_{\overline{u_k}}H \partial_{u_k} K - \partial_{u_k}H \partial_{\overline{u_k}} K.
  $$
For all $s\in \mathbb{R}$, we define the $h^s$ norm on $\mathbb{C}^{\mathcal{T}_M}$ by
$$
\forall u \in \mathbb{C}^{\mathcal{T}_M}, \quad \|u\|_{h^s}^2 := \sum_{k = (\ell,m)\in \mathcal{T}_M} \langle \ell \rangle^{2s} |u_k|^2
$$

\subsection{Multilinear estimates}

In this paragraph, we establish multilinear estimates for Hamiltonians which are homogeneous polynomials on $\mathbb{C}^{\mathcal{T}_M}$.
\begin{definition}[Space $\mathscr{H}_M^r$]\label{defH} Being given $M\geq 0$ and $r\geq 2$, $\mathscr{H}_M^r$ denotes the space of real valued homogeneous polynomial of degree $r$ on the real vector space $\mathbb{C}^{\mathcal{T}_M}$.
\end{definition}

\begin{remark} By definition, every homogeneous polynomial $H\in \mathscr{H}_M^r$ admits a unique decomposition of the form
\begin{equation*}
H(u) = \sum_{\sigma \in \{-1,1\}^r} \sum_{\mathbf{k}\in \mathcal{T}_M^r } H_{\mathbf{k}}^{\sigma} u_{k_1}^{\sigma_1} \dots u_{k_r}^{\sigma_r}
\end{equation*}
where $(H_{\mathbf{k}}^\sigma)_{ (\mathbf{k},\sigma)\in \mathcal{T}_M^r \times \{-1,1\}^r}$ is a sequence of complex numbers satisfying the reality condition
\begin{equation}
\label{eq:def_real}
H_{\mathbf{k}}^{-\sigma} = \overline{H_{\mathbf{k}}^{\sigma}}
\end{equation}
and the symmetry condition 
\begin{equation}
\label{eq:def_sym_cond}
\forall \phi \in \mathfrak{S}_r, \ H_{k_1,\dots,k_r}^{\sigma_1,\dots,\sigma_r} =  H_{k_{\phi_1},\dots,k_{\phi_r}}^{\sigma_{\phi_1},\dots,\sigma_{\phi_r}}.
\end{equation}
\end{remark}

We endow this space of polynomials with two unusual norms $\| \cdot \|_{\mathscr{H}}$ and $\| \cdot \|_{\mathscr{C}}$. Roughly speaking, in our Birkhoff normal form process, the terms of the Taylor expansion of the Hamiltonian are controlled with the $\mathscr{H}$-norm whereas the solutions to cohomological equations are controlled with a $\mathscr{C}$-norm (because they enjoy better properties).
\begin{definition}[Norms $\| \cdot \|_{\mathscr{H}}$ and $\| \cdot \|_{\mathscr{C}}$] Let $M\geq 0$, $r \geq 2$ and $H,\chi \in \mathscr{H}_M^r$, we set
\begin{equation}
\label{eq:norm_HM}
\| H \|_{\mathscr{H}} := \max_{\sigma \in \{-1,1\}^r } \max_{\mathbf{k} \in  \mathcal{T}_M^r } |H_{\mathbf{k}}^{\sigma}| \sqrt{\langle \ell_1 \rangle \cdots \langle \ell_r \rangle } \sqrt{\Upsilon(\mathbf{k})}
\end{equation}
and
\begin{equation}
\label{eq:norm_CMd}
\| \chi \|_{\mathscr{C}} := \max_{\sigma \in \{-1,1\}^r } \max_{\mathbf{k} \in  \mathcal{T}_M^r } |\chi_{\mathbf{k}}^{\sigma}| \langle \sigma_1 \ell_1 + \cdots + \sigma_r \ell_r \rangle \sqrt{\langle \ell_1 \rangle \cdots \langle \ell_r \rangle } \sqrt{ \Upsilon( \mathbf{k})}
\end{equation}
where $k_j =: (\ell_j,m_j)$ for all $j \in \llbracket 1,r\rrbracket$ and $\Upsilon$ is defined by \eqref{def:upsilon}.
\end{definition}

As we shall see in this section, these nonstandard norms are well behaved with the symplectic operations (Poisson bracket, gradient) that are used when performing a Birkhoff normal form procedure in Theorem~\ref{thm_NF}. One reason for these nice properties is the fact that they involve an extra regularity factor $\Upsilon( \mathbf{k})$ which only depends on the largest simple index $k_j=(\ell_j,m_j)$ of $\mathbf{k}$. Despite their unusual definition, these norms can be implemented in our normal form argument as this exponent appears naturally in the multilinear estimate of Theorem~\ref{t:proba}. See for instance~\eqref{eq:toutcapourca} below.

Let us now turn to the nice properties enjoyed by these norms. They provide the following continuity estimate for the Poisson bracket :
\begin{proposition}\label{poisson} Let $r,r'\geq 2$ and $M\geq 2$. For all $H\in \mathscr{H}_M^{r'}$ and all $\chi \in \mathscr{H}_{M}^r$, their Poisson bracket $ \{\chi,H\}$ is a homogeneous polynomial of degree $r+r'-2$ (i.e. $\{\chi,H\} \in \mathscr{H}_M^{r+r'-2}$) enjoying the bound 
$$
\| \{\chi,H\} \|_{\mathscr{H}} \lesssim_{r,r'} \log M  \ \| H\|_{\mathscr{H}} \|\chi\|_{\mathscr{C}}.
$$
\end{proposition}
\proof
By definition of the Poisson bracket, we have 
 \begin{equation}
 \label{eq_un_lapin} 
\{ \chi,H\}(u)  = 
  2i \sum_{\mathfrak{K}\in \mathcal{T}_M} \partial_{\bar{u}_\mathfrak{K}}\chi(u) \partial_{u_\mathfrak{K}} H(u) - \partial_{u_\mathfrak{K}}\chi(u) \partial_{\bar{u}_\mathfrak{K}} H(u).
 \end{equation}
Since the coefficients of $H$ and $K$ are symmetric (i.e. satisfy \eqref{eq:def_sym_cond}), we have
\begin{equation}
\label{eq_a_tue_un_chasseur}
\partial_{\bar{u}_\mathfrak{K}} \chi \partial_{u_\mathfrak{K}} H = r r'  \sum_{\substack{\sigma \in \{-1,1\}^{r-1}\\ \sigma' \in \{-1,1\}^{r'-1} }} \sum_{\substack{\mathbf{k}\in  \mathcal{T}_M^{r-1} \\ \mathbf{k}'\in  \mathcal{T}_M^{r'-1}  }}    \chi_{\mathbf{k},\mathfrak{K}}^{\sigma,-1} u_{k_1}^{\sigma_1} \dots u_{k_{r-1}}^{\sigma_{r-1}}  H_{\mathbf{k}',\mathfrak{K}}^{\sigma',1} u_{k'_1}^{\sigma'_1} \dots u_{k'_{r'-1}}^{\sigma'_{r'-1}}.
\end{equation}
Obviously, $\{ \chi,H\}$ defines an homogeneous polynomial of degree $r+r'-2$. Hence, we need to verify the reality condition~\eqref{eq:def_real} and the upper bound on the $\mathscr{H}$-norm. For the latter, we begin by estimating $\sum_\mathfrak{K} \chi_{\mathbf{k},\mathfrak{K}}^{\sigma,-1} H_{\mathbf{k}',\mathfrak{K}}^{\sigma',1}$. By \eqref{eq:norm_HM} and \eqref{eq:norm_CMd}, denoting $\mathbf{k}\in  \mathcal{T}_M^{r-1}$, $\mathbf{k}'\in  \mathcal{T}_M^{r'-1}$ , $\mathbf{k}''=(\mathbf{k},\mathbf{k}')$ and $r''=r+r'-2$, we have
\begin{equation}
\begin{split}
\label{po1}
\sum_{\mathfrak{K}\in \mathcal{T}_M}| \chi_{\mathbf{k},\mathfrak{K}}^{\sigma,-1} H_{\mathbf{k}',\mathfrak{K}}^{\sigma',1}| &\leq \frac{\| H\|_{\mathscr{H}} \|\chi\|_{\mathscr{C}}}{\sqrt{\langle \ell_1 \rangle \cdots \langle \ell_{r-1} \rangle \langle \ell'_1 \rangle \cdots \langle \ell'_{r'-1} \rangle }}\times\\
&\sum_{\mathfrak{K}=(\mathfrak{l},\mathfrak{m})\in \mathcal{T}_M}\frac{1}{\langle \mathfrak{l} \rangle\langle \sigma_1 \ell_1 + \cdots + \sigma_{r-1} \ell_{r-1}-\mathfrak{l}\rangle\sqrt{ \Upsilon( \mathbf{k},\mathfrak{K}) \Upsilon( \mathbf{k}',\mathfrak{K})}}.
\end{split}
\end{equation}
We claim that for all $\mathfrak{K}  \in \mathcal{T}_M$ we have 
\begin{equation}\label{po2}\Upsilon(\mathbf{k},\mathbf{k}')\leq \Upsilon( \mathbf{k},\mathfrak{K}) \Upsilon( \mathbf{k}',\mathfrak{K}).\end{equation}
 Indeed, if $\Upsilon(\mathbf{k},\mathbf{k}')=1$ the inequality is trivial so we can assume that 
 \begin{itemize}
 \item either there exists $1\leq i\leq r-1$  such that $\Upsilon(\mathbf{k},\mathbf{k}')=\langle \ell_i\rangle $, $k_j\neq k_i$ for $1\leq j\leq r-1$ with $j\neq i$ and $k'_{j'}\neq k_{i}$ for $1\leq j'\leq r'-1$, 
 \item or there exists  $1\leq i'\leq r'-1$ such that $\Upsilon(\mathbf{k},\mathbf{k}')=\langle \ell'_{i'}\rangle $, $k'_{j'}\neq k_{i'}$ for $1\leq j'\leq r'-1$ with $j'\neq i'$ and $k_{j}\neq k'_{i'}$ for $1\leq j\leq r-1$.
 \end{itemize}
  By symmetry of the problem, let us assume the former and  let $\mathfrak{K} = (\mathfrak{l},\mathfrak{m}) \in \mathcal{T}_M $.\\
  If $\Upsilon( \mathbf{k},\mathfrak{K})\geq \langle \ell_i\rangle=\Upsilon(\mathbf{k},\mathbf{k}')$ then \eqref{po2} holds true trivially. 
 So let us assume that  $\Upsilon( \mathbf{k},\mathfrak{K})< \langle \ell_i\rangle$. This implies that $\mathfrak{K}=k_i$ (if not $\Upsilon( \mathbf{k},\mathfrak{K})$ is the maximum of a list of numbers including  $\langle \ell_i\rangle$ ). But then, if $\Upsilon( \mathbf{k}',\mathfrak{K})\geq \langle \mathfrak{l} \rangle$, we deduce $\Upsilon( \mathbf{k}',\mathfrak{K})\geq \langle \ell_i\rangle= \Upsilon(\mathbf{k},\mathbf{k}')$ which in turn implies \eqref{po2}. Thus it remains to consider the case $\Upsilon( \mathbf{k}',\mathfrak{K})< \langle \mathfrak{l} \rangle$ which leads to the existence of $1\leq j'\leq r'-1$ such that $k_{j'}=\mathfrak{K}$ (if not $\Upsilon( \mathbf{k}',\mathfrak{K})$ is the maximum of a list of numbers including  $\langle \mathfrak{l} \rangle$). Therefore $k_i=k_{j'}$ which contradicts the definition of $i$.

Implementing~\eqref{po2} in~\eqref{po1} and denoting $a=\sigma_1 \ell_1 + \cdots + \sigma_{r-1} \ell_{r-1} $, one is left with estimating
\begin{align}\label{po3}
\sum_{\mathfrak{K} = (\mathfrak{l},\mathfrak{m}) \in \mathcal{T}_M}\frac{1}{\langle \mathfrak{l} \rangle\langle \sigma_1 \ell_1 + \cdots + \sigma_{r-1} \ell_{r-1}-\mathfrak{l} \rangle}&\leq 4 \sum_{\mathfrak{l} =0}^M \frac1{\sqrt{1+(a-\mathfrak{l} )^2}}\\ \nonumber &\leq 4\sum_{j=-a}^{M-a} \frac1{\sqrt{1+j^2}}\leq 8\sum_{j=0}^{M} \frac1{\sqrt{1+j^2}}
\lesssim \log M
\end{align}
independently of the value of $a$.

Inserting \eqref{po2} and \eqref{po3} in \eqref{po1}, we get uniformly with respect to $\sigma,\sigma',k,k'$
\begin{equation}\label{po4}
\sum_{\mathfrak{K}  \in \mathcal{T}_M}| \chi_{\mathbf{k},\mathfrak{K} }^{\sigma,-1} H_{\mathbf{k}',\mathfrak{K} }^{\sigma',1}| \lesssim \log M\frac{\| H\|_{\mathscr{H}} \|\chi\|_{\mathscr{C}}}{ \sqrt{ \Upsilon( \mathbf{k},\mathbf{k}')}\sqrt{\langle \ell_1 \rangle \cdots \langle \ell_{r-1} \rangle \langle \ell'_1 \rangle \cdots \langle \ell'_{r'-1} \rangle }}.
\end{equation}
Then, denoting $r''=r+r'-2$, $\mathbf{k}''=(\mathbf{k},\mathbf{k}')$ and $\sigma''=(\sigma,\sigma')$, we define
$$
M_{\mathbf{k}''}^{\sigma''} := 2 i r r' \sum_{\mathfrak{K}  \in \mathcal{T}_M} \chi_{\mathbf{k},\mathfrak{K} }^{\sigma,-1} H_{\mathbf{k}',\mathfrak{K} }^{\sigma',1} - \chi_{\mathbf{k},\mathfrak{K} }^{\sigma,1} H_{\mathbf{k}',\mathfrak{K} }^{\sigma',-1} \quad \mathrm{and} \quad P_{\mathbf{k}''}^{\sigma''} = \frac1{r'' !} \sum_{\rho \in \mathfrak{S}_{r''}} M_{\mathbf{k}''\circ \rho}^{\sigma''\circ \rho}.
$$
By definition, $P(u) = \{ \chi, H\}(u) $ and the estimates \eqref{po4} proves that 
$$
\| P\|_{\mathscr{H}} \lesssim r r' \log M\| H\|_{\mathscr{H}} \|\chi\|_{\mathscr{C}}.
$$ 
Finally, the coefficients of $P$ are obviously symmetric and, by a direct calculation, we verify that they satisfy the reality condition \eqref{eq:def_real}.

\endproof

We now study the vector field on $\mathbb{C}^{\mathcal{T}_M}$ associated with an Hamiltonian in $\mathscr{H}_M^r$.

\begin{lemma}\label{grad-} Let $M\geq 2$ and $r\geq 2$. For all $H\in \mathscr{H}_M^r$, $H$ is a real valued  smooth map on $\mathbb{C}^{\mathcal{T}_M}$ which enjoys the bounds
$$
\forall u\in \mathbb{C}^{\mathcal{T}_M}, \quad \| \nabla H(u) \|_{h^{-1/2}} \lesssim_r (\log (M) )^{r/2} \| H \|_{\mathscr{H}} \| u\|_{h^{1/2}}^{r-1} 
$$
\end{lemma}
\proof As a polynomial (of finitely many variables), any Hamiltonian $H\in \mathscr{H}_M^r$ is a smooth map on $\mathbb{C}^{\mathcal{T}_M}$. We aim at bounding the norm by duality. To that aim, we fix $v\in  \mathbb{C}^{\mathcal{T}_M}$ and we need to estimate $|(\nabla H(u),v)_{\ell^2}|$. Since the coefficients of $H$ are symmetric, we then write
\begin{align*}|(\nabla H(u),v)_{\ell^2}|&\leq r\|H\|_{\mathscr{H}}\sum_{\sigma \in \{-1,1\}^r} \sum_{\mathbf{k}\in \mathcal{T}_M^r } \frac{ |u_{k_1}^{\sigma_1}|}{\langle \ell_1\rangle^{\frac12}} \dots \frac{|v_{k_r}^{\sigma_r}|}{\langle \ell_r\rangle^{\frac12}}\\
&\leq r2^r\|H\|_{\mathscr{H}} \sum_{\mathbf{k}\in \mathcal{T}_M^r } \frac{ \langle \ell_1\rangle^{\frac12}| u_{k_1}|}{\langle \ell_1\rangle} \dots \frac{\langle \ell_r\rangle^{\frac12}|v_{k_r}|}{\langle \ell_r\rangle}\\
&\leq r2^r \|H\|_{\mathscr{H}}\|u\|_{h^{1/2}}^{r-1} \|v\|_{h^{1/2}}\Big(\sum_{k=(\ell,m)\in \mathcal{T}_M } \frac1{\langle \ell \rangle^2} \Big)^{r/2}\\
& \lesssim_r (\log (M) )^{r/2} \| H \|_{\mathscr{H}} \|u\|_{h^{1/2}}^{r-1} \|v\|_{h^{1/2}}.
\end{align*}
Then by duality we obtain
$$
\| \nabla H(u) \|_{h^{-1/2}} \lesssim_r (\log (M) )^{r/2} \| H \|_{\mathscr{H}} \| u\|_{h^{1/2}}^{r-1}.
$$
\endproof

The $\mathscr{C}$-norm provides a better estimate on the gradient:
\begin{lemma}\label{grad+} Let $M\geq 2$, $r\geq 2$. For all $\chi\in \mathscr{H}_{M}^r$ and all $u\in \mathbb{C}^{\mathcal{T}_M}$, we have the bounds
\begin{equation}\label{estim-grad-chi}
\| \nabla \chi(u) \|_{h^{1/2}} \lesssim_r (\log (M) )^{(r-1)/2} \| \chi \|_{\mathscr{C}} \| u\|_{h^{1/2}}^{r-1} 
\end{equation}
and
\begin{equation}\label{estim-dgrad-chi}
 \| \mathrm{d}\nabla \chi(u) \|_{\mathscr{L}(h^{1/2})} \lesssim_r  (\log (M) )^{(r-1)/2} \| \chi \|_{\mathscr{C}} \| u\|_{h^{1/2}}^{r-2}.
\end{equation}
\end{lemma}
\proof Without loss of generality, we assume that $\| \chi \|_{\mathscr{C}}  = 1$.  We aim at proving \eqref{estim-grad-chi} by duality i.e., for every $v\in  \mathbb{C}^{\mathcal{T}_M}$, we want to estimate $|(\nabla \chi(u),v)_{\ell^2}|$. We denote $\tilde u_k=\langle \ell \rangle^{\frac12}|u_k|$ and $\tilde v_k=\langle \ell\rangle^{-\frac12}|v_k|$ for all $k=(\ell,m)\in \mathcal{T}_M$ in such way $\|\tilde u\|_{\ell^2}=\|u\|_{h^{1/2}}$ and $\|\tilde v\|_{\ell^2}=\|v\|_{h^{-1/2}}$. Since the coefficient of $\chi$ are symmetric, we have
\begin{equation}\label{e:differential} 
(\nabla \chi(u),v)_{\ell^2}  =   r\sum_{\sigma \in \{-1,1\}^{r}} \sum_{\mathbf{k}\in \mathcal{T}_M^{r} }  \chi_{\mathbf{k}}^{\sigma} u_{k_1}^{\sigma_1} \dots u_{k_{r-1}}^{\sigma_{r-1}} v_{k_r}^{\sigma_r}.
\end{equation}
Then by applying the triangular inequality, we get
\begin{equation*}
|(\nabla \chi(u),v)_{\ell^2}|\leq 2r \sum_{\sigma \in \{-1,1\}^{r-1}}\sum_{\mathbf{k}\in \mathcal{T}_M^r } \frac1{\langle \sigma_1 \ell_1 + \cdots + \sigma_{r-1} \ell_{r-1}-\ell_r\rangle \sqrt{ \Upsilon( \mathbf{k})}}\frac{ \tilde u_{k_1}}{\langle \ell_1\rangle} \dots \frac{\tilde u_{k_{r-1}}}{\langle \ell_{r-1}\rangle}\tilde v_{k_r}.
\end{equation*}
At this stage, we notice that, for all $\mathbf{k}\in\mathcal T_M^r$,  we have $\Upsilon( \mathbf{k})\geq \Upsilon'( \mathbf{k})$ where $\Upsilon'( \mathbf{k})=1$ except when $k_j\neq k_{r}$ for all $j=1,\cdots,r-1$ and in that case $\Upsilon'( \mathbf{k})=\langle \ell_{r}\rangle$. Thus 
\begin{align*}
|(\nabla \chi(u),v)_{\ell^2}|&\leq 2r \sum_{\sigma \in \{-1,1\}^{r-1}}\sum_{\mathbf{k}\in \mathcal{T}_M^r } \frac1{\langle \sigma_1 \ell_1 + \cdots + \sigma_{r-1} \ell_{r-1}-\ell_r\rangle \sqrt{\langle \ell_{r}\rangle}}\frac{ \tilde u_{k_1}}{\langle \ell_1\rangle} \dots \frac{\tilde u_{k_{r-1}}}{\langle \ell_{r-1}\rangle}\tilde v_{k_r}\\
&+2r \sum_{\sigma \in \{-1,1\}^{r-1}}\sum_{\substack{\mathbf{k}\in \mathcal{T}_M^r\\ \exists  1\leq i\leq r-1:\ k_r=k_i}} \frac1{\langle \sigma_1 \ell_1 + \cdots + \sigma_{r-1} \ell_{r-1}-\ell_r\rangle}\frac{ \tilde u_{k_1}}{\langle \ell_1\rangle} \dots \frac{\tilde u_{k_{r-1}}}{\langle \ell_{r-1}\rangle}\tilde v_{k_r}\\
&=2r(\Sigma_1+\Sigma_2). 
\end{align*}
First we estimate $\Sigma_1$
$$
\Sigma_1= \sum_{\sigma \in \{-1,1\}^{r-1}}\sum_{\mathbf{k}\in \mathcal{T}_M^{r}}
\frac{ \tilde u_{k_1}}{\langle \ell_1\rangle} \dots \frac{\tilde u_{k_{r-1}}}{\langle \ell_{r-1}\rangle}
 \frac{\tilde v_{k_r}}{\langle \sigma_1 \ell_1 + \cdots + \sigma_{r-1} \ell_{r-1}-\ell_r\rangle \langle \ell_r\rangle^{\frac12}}.$$
 We notice that 
$$ 
\sum_{k=(\ell,m) \in \mathcal{T}_M } \frac1{\langle \ell+a\rangle^2 \langle \ell \rangle} = \sum_{\ell=0}^M \frac{2\ell+1}{\langle \ell \rangle}  \frac1{\langle \ell+a\rangle^2 }  \leq \sum_{j\in\mathbb Z}\frac4{\langle j\rangle^2}\lesssim 1
$$
uniformly with respect to $a\in\mathbb R$ and 
$$ \sum_{k = (\ell,m) \in \mathcal{T}_M } \frac1{\langle \ell\rangle^2}\lesssim \log (M). $$
Thus by Cauchy-Schwarz we get
\begin{align*}
\Sigma_{1}&\lesssim_r \|u\|_{h^{1/2}}^{r-1}(\log (M) )^{(r-1)/2} \|v\|_{h^{-1/2}}.
\end{align*}
It remains to estimate $\Sigma_2$. We can assume without lost of generality, but paying an extra factor $r$, that  $ k_{r-1}= k_r$. Then, by Cauchy-Schwarz, we get
\begin{align*}
\Sigma_{2}&\leq r 2^{r-1} \sum_{k_{r-1}\in \mathcal{T}_M } \tilde  u_{k_{r-1}}\tilde v_{k_{r-1}}  \sum_{k = (\ell,m)\in \mathcal{T}_M^{r-2} } \frac{ \tilde u_{k_1}}{\langle \ell_1\rangle} \dots \frac{\tilde u_{k_{r-2}}}{\langle \ell_{r-2}\rangle}\\
&\lesssim_r \|u\|_{h^{1/2}}^{r-1}(\log (M) )^{(r-2)/2} \|v\|_{h^{-1/2}}. \end{align*}
Putting together the estimates of $\Sigma_1$ and $\Sigma_{2}$ we conclude that, for all $v\in \in \mathbb{C}^{\mathcal{T}_M}$,
$$|(\nabla \chi(u),v)|\lesssim_r (\log (M) )^{(r-1)/2} \|v\|_{h^{-1/2}}\|u\|_{h^{1/2}}^{r-1}$$
which in turn implies \eqref{estim-grad-chi}.\\
To prove \eqref{estim-dgrad-chi} we just notice that since $\nabla \chi(u)$ is an homogeneous polynomial, it can be viewed as the trace of a $(r-1)$-linear map on $\mathbb{C}^{\mathcal{T}_M}$: $\nabla \chi(u)=F(u,\cdots,u)$ with $F$ that can be expressed using~\eqref{e:differential}. Thus, following the above proof, $F$ satisfies
$$\| F(u^{(1)},\cdots,u^{(r-1)})\|_{h^{1/2}}\lesssim_r (\log (M) )^{(r-1)/2} \|u^{(1)}\|_{h^{1/2}}\cdots \|u^{(r-1)}\|_{h^{1/2}}.$$
Then, since $d \nabla \chi(u)(v)=F(v,u,\cdots,u)+\cdots+F(u,\cdots,u,v)$, we deduce \eqref{estim-dgrad-chi}.
\endproof

Thanks to a standard duality argument, we rewrite the estimate \eqref{estim-dgrad-chi} in a negative Sobolev space.
\begin{corollary} Let $M\geq 2$, $r\geq 2$. For all $\chi\in \mathscr{H}_{M}^r$ and $u\in \mathbb{C}^{\mathcal{T}_M}$, we have
\begin{equation}\label{eq_chocolat_chaud}
 \| \mathrm{d}\nabla \chi(u) \|_{\mathscr{L}(h^{-1/2})}  \lesssim_r  (\log (M) )^{(r-1)/2} \| \chi \|_{\mathscr{C}} \| u\|_{h^{1/2}}^{r-2}.
\end{equation}
\end{corollary}
\proof
By duality  we have
$$
\sup_{\substack{v\in \mathbb{C}^{\mathcal{T}_M} \\ \| v\|_{h^{-1/2}}\leq 1}}   \| \mathrm{d}\nabla \chi(u)(v) \|_{h^{-1/2}} =  \sup_{\substack{v\in \mathbb{C}^{\mathcal{T}_M} \\ \| v\|_{h^{-1/2}}\leq 1}} \sup_{\substack{w\in \mathbb{C}^{\mathcal{T}_M} \\ \| w\|_{h^{1/2}}\leq 1}}   (w, \mathrm{d} \nabla \chi(u)(v) )_{\ell^2}  .
$$
Then by applying the Schwarz theorem we have
\begin{multline*}
(w, \mathrm{d} \nabla \chi(u)(v) )_{\ell^2}   = \mathrm{d}[ (w,\nabla \chi(u) )_{\ell^2} ](v)   = \mathrm{d}[ \mathrm{d} \chi(u)(w)  ](v) =  \mathrm{d}^2 \chi(u)(w)(v)  \\= \mathrm{d}^2 \chi(u)(v)(w) = \mathrm{d}[ (v, \nabla \chi(u) )_{\ell^2} ](w)  = (v, \mathrm{d} \nabla \chi(u)(w) )_{\ell^2}.
\end{multline*}
Therefore
\begin{align*}
\sup_{\substack{v\in \mathbb{C}^{\mathcal{T}_M} \\ \| v\|_{h^{-1/2}}\leq 1}}   \| \mathrm{d} \nabla \chi(u)(v) \|_{h^{-1/2}}& = \! \! \sup_{\substack{w\in \mathbb{C}^{\mathcal{T}_M} \\ \| w\|_{h^{1/2}}\leq 1}}   \sup_{\substack{v\in \mathbb{C}^{\mathcal{T}_M} \\ \| v\|_{h^{-1/2}}\leq 1}}   (v, \mathrm{d} \nabla \chi(u)(w) )_{\ell^2}  = \! \!  \sup_{\substack{w\in \mathbb{C}^{\mathcal{T}_M} \\ \| w\|_{h^{1/2}}\leq 1}}     \| \mathrm{d} \nabla \chi(u)(w) \|_{h^{1/2}} 
\\&= \|\mathrm{d} \nabla \chi(u)\|_{\mathscr{L}(h^{1/2})}.
\end{align*}
As a consequence, \eqref{eq_chocolat_chaud} is just a corollary of the estimate \eqref{estim-dgrad-chi}.
\endproof

Finally we define the flow associated with an Hamiltonian in $\mathscr{H}_{M}^r$:
\begin{proposition}
\label{prop_ham_flow} Let $M\geq 2$, $r\geq 3$ and $\chi \in \mathscr{H}_{M}^r$. There exist 
\begin{equation}\label{eps0}
\varepsilon_0 \gtrsim_{r} \big( (\log (M) )^{(r-1)/2} \|\chi \|_{\mathscr{C}} \big)^{-1/(r-2)}
\end{equation}
 and a smooth map
$$
\Phi_\chi : \left\{ \begin{array}{cll} [-1,1] \times B_{h^{1/2}(\mathbb{C}^{\mathcal{T}_M})}(0,\varepsilon_0) &\to& \mathbb{C}^{\mathcal{T}_M} \\ (t,u) &\mapsto& \Phi_\chi^t(u) \end{array} \right.
$$
solving the equation
\begin{equation}\label{flow}
-i\partial_t \Phi_\chi = ( \nabla \chi)\circ \Phi_\chi,
\end{equation}
and such that for all $t\in [-1,1]$, $\Phi_\chi^t$ is symplectic, close to the identity
\begin{equation}
\label{eq_dindon}
\forall u \in  B_{h^{1/2}(\mathbb{C}^{\mathcal{T}_M})}(0,\varepsilon_0), \quad \| \Phi_\chi^t u - u \|_{h^{1/2}} \leq \Big(\frac{\| u \|_{h^{1/2}}}{\varepsilon_0} \Big)^{r-2} \| u \|_{h^{1/2}},
\end{equation}
invertible
\begin{equation}
\label{eq_invertibility}
\|\Phi_\chi^{t} (u) \|_{h^{1/2}} < \varepsilon_0 \quad \Rightarrow \quad  \Phi_\chi^{-t}\circ  \Phi_\chi^{t} (u) = u.
\end{equation}
Moreover, its differential enjoys the estimate
\begin{equation}
\label{eq_petanque}
\forall u \in  B_{h^{1/2}(\mathbb{C}^{\mathcal{T}_M})}(0,\varepsilon_0),\forall \sigma \in \{-1,1\},\quad \| \mathrm{d} \Phi_\chi^t (u)\|_{\mathscr{L}(h^{\sigma/2})}\leq 2.
\end{equation}
\end{proposition}

\proof
We note that \eqref{flow} is an ODE associated with the smooth vector field $X_{\chi}=i\nabla\chi$ and therefore we deduce from the Cauchy--Lipschitz Theorem that the flow $\Phi_\chi^t(u)$ is locally well defined for every $u\in \mathbb{C}^{\mathcal{T}_M}$ on some maximal interval $(T_-(u),T_+(u))$ containing $0$. Let us first show that, if $\|u\|_{h^{1/2}}=\varepsilon$ is small enough, then the solution is defined up to time $1$, equivalently $T_+(u)\geq 1$. To see this, we set 
$$t_0:=\sup\left\{t\in [0,T_+(u)):\ \forall 0\leq s\leq t, \|\Phi_\chi^s(u)\|_{h^{1/2}}< 2\varepsilon\right\}>0.$$
In the case where $T_+(u)<\infty$, we note that $t_0<T_+(u)$ by the maximality of the interval of definition and we can verify that $t_0\geq 1$ provided $\varepsilon$ is chosen small enough. Indeed, if $t_0<1$, then we can write
\begin{align*}\label{dindon}\varepsilon\leq\|\Phi_\chi^{t_0}(u)-u\|_{h^{1/2}}&\leq \int_0^{t_0} \|( \nabla \chi)\circ \Phi_\chi^s(u)\|_{h^{1/2}}  ds\\&\leq C_r^{-(r-2)} t_0 (\log (M) )^{(r-1)/2} \varepsilon^{r-1}\|\chi \|_{\mathscr{C}},\end{align*}
for some constant $0<C_r\leq 1$ depending only on $r$ coming from~\eqref{estim-grad-chi}. From this, we infer
$$\varepsilon^{-1}
\left((\log (M) )^{(r-1)/2} \|\chi \|_{\mathscr{C}}\right)^{-\frac{1}{r-2}}\leq C_r^{-1} |t_0|^{\frac{1}{r-2}}.
$$
Thus, as long as $\varepsilon \leq C_r\left((\log (M) )^{(r-1)/2} \|\chi \|_{\mathscr{C}}\right)^{-\frac{1}{r-2}}$, we find that $t_0\geq 1$ and that the flow is well defined up to time $t=1$. The same holds in negative times. We now fix 
$$
\varepsilon_0:=\frac{ C_r}{2}\left((\log (M) )^{(r-1)/2} \|\chi \|_{\mathscr{C}}\right)^{-\frac{1}{r-2}}
$$ so that $t_0\geq 1$ for every $\|u\|_{h^{1/2}}=\varepsilon<\varepsilon_0$. Since $\Phi_\chi^t(u)$ is the flow associated with an Hamiltonian vector field, it is symplectic and invertible and we are left with the proof of~\eqref{eq_dindon} and~\eqref{eq_petanque}. For the former, we write as above, for $-1\leq t\leq 1$,
\begin{align*}\|\Phi_\chi^{t}(u)-u\|_{h^{1/2}}&\leq \Big|\int_0^{t} \|( \nabla \chi)\circ \Phi_\chi^s(u)\|_{h^{1/2}}  \mathrm{d}s\Big|\\&\leq C_r^{-(r-2)}\|\chi \|_{\mathscr{C}}(\log (M) )^{(r-1)/2} \|u\|_{h^{1/2}}^{r-1}\leq \left(\frac{\|u\|_{h^{1/2}}}{\varepsilon_0}\right)^{r-2}\|u\|_{h^{1/2}}.\end{align*}
It now remains to prove~\eqref{eq_petanque}. Up to decreasing the value of $\varepsilon_0$ a little bit (by a factor depending only on $r$), we can proceed as above by appealing \eqref{estim-dgrad-chi} and \eqref{eq_chocolat_chaud} and by writing
$$\mathrm{d} \Phi_\chi^t (u)=\text{Id}+\int_0^t \mathrm{d}\nabla\chi(\Phi_\chi^s(u))\circ\mathrm{d}\Phi_\chi^s(u)ds.$$


\endproof

\section{Birkhoff normal form}

In this section, we aim at  describing a procedure that allows to simplify, close to $u=0$, Hamiltonians on $\mathbb{C}^{\mathcal{T}_M}$ that are of the form
$$
H(u):=\frac12 \sum_{k\in \mathcal{T}_M} \omega_k |u_k|^2+P(u),
$$
where $P \in \mathscr{H}_M^p$. In other words, we will write a Birkhoff normal form for $H$ which means that, up to conjugation by a symplectomorphism and up to a small remainder term, $P$ can be replaced by a term Poisson commuting with the super actions composing the leading part of $H$:
$$\forall \ell \geq 0,\quad J_\ell(u) = \sum_{ m=-\ell}^{\ell} |u_{(\ell,m)}|^2.$$
This will be used in Section~\ref{sec-proofs} to put \eqref{eq:KG} into a Birkhoff normal form and to prove or main theorem.
From now on, we fix an integer $p \geq 3$ (the degree of the nonlinearity of \eqref{eq:KG}) and $\mu>0$ (the mass of \eqref{eq:KG}) making the frequencies ($\omega_{(\ell,m)} = \sqrt{\ell (\ell+1) + \mu}$)  non-resonant (in the sense of Proposition \ref{prop:strong_nr}). Our precise Birkhoff normal form statement reads as follows:

\begin{theorem} 
\label{thm_NF}
Let $a>0$, $C_p>0$ and $r\geq 1$. Then, there exist $\beta>1$ (independent of the choice of $\mu$) and $C>1$ such that the following holds.

For every $M\geq 2$, $N\geq 1$ and every polynomial Hamiltonian of the form $H: \mathbb{C}^{\mathcal{T}_M} \to \mathbb{R}$
$$
H = Z_2 + P^{(p)} \quad \mathrm{where} \quad Z_2(u) = \frac12 \sum_{k\in \mathcal{T}_M} \omega_k |u_k|^2,\quad P^{(p)} \in \mathscr{H}_M^p,\quad \| P^{(p)} \|_{\mathscr{H}} \leq C_p B^{a}
$$
with $B = \max( \log M , N)$, one can find $\varepsilon_2 \geq (C B^\beta)^{-1}$ and two smooth symplectic maps $\tau^{(0)}$ and $\tau^{(1)}$ making the following diagram to commute




\begin{equation}
\label{mon_beau_diagram_roi_des_forets}
\xymatrixcolsep{5pc} \xymatrix{  B_{h^{1/2}(\mathbb{C}^{\mathcal{T}_M})}(0, \varepsilon_2 ) \ar[r]^{ \tau^{(0)} }
 \ar@/_1pc/[rr]_{ \mathrm{id}_{\mathbb{C}^{\mathcal{T}_M}}} &  B_{h^{1/2}(\mathbb{C}^{\mathcal{T}_M})}(0,2\, \varepsilon_2)  \ar[r]^{ \hspace{1cm} \tau^{(1)} }  & 
 \mathbb{C}^{\mathcal{T}_M} } 
\end{equation}
and close to the identity
\begin{equation}
\label{eq:taucloseid}
\forall \nu\in \{0,1\}, \ \|u\|_{{h}^{1/2}} <2^{\nu}\varepsilon_2 \;\; \Rightarrow \;\;  
\|\tau^{(\nu)}(u)-u\|_{h^{1/2}} \leq  \left( \frac{\|u\|_{h^{1/2}}}{2^{\nu}\varepsilon_2} \right)^{p-2} \|u\|_{h^{1/2}}
\end{equation}
such that, on $ B_{h^{1/2}(\mathbb{C}^{\mathcal{T}_M})}(0,2\varepsilon_2)$, $H \circ \tau^{(1)}$ admits the decomposition
\begin{equation}\label{eq:NF}
H \circ \tau^{(1)} = Z_2+ Q^{\leq N}_{\mathrm{res}} + R
\end{equation}
where $Q^{\leq N}_{\mathrm{res}} : \mathbb{C}^{\mathcal{T}_M} \to \mathbb{R}$ is a polynomial of degree $r+p-1$ commuting with the low super-actions
\begin{equation}
\label{eq:ca_commte}
\forall \ell \in \mathbb{N}, \ \langle \ell \rangle \leq N \ \Rightarrow \ \{ J_\ell , Q^{\leq N}_{\mathrm{res}} \} = 0.
\end{equation}
Moreover, 
the remainder term $R$ is a smooth function on $B_{h^{1/2}(\mathbb{C}^{\mathcal{T}_M})}(0,2\varepsilon_2)$ satisfying
$$
\|\nabla R(u)\|_{h^{-1/2}} \leq C B^{\beta} \|u\|_{h^{1/2}}^{r+p-1},
$$
and, 
for all $\nu \in\{0,1\}$, we have the bounds
\begin{equation}
\label{eq:dtau}
\|\mathrm{d} \tau^{(\nu)}(u) \|_{\mathscr{L}(h^{1/2})} \leq 2^{r} \quad \mathrm{and} \quad \|\mathrm{d} \tau^{(\nu)}(u) \|_{\mathscr{L}(h^{-1/2})} \leq 2^{r}.
\end{equation}
\end{theorem}

\begin{proof}
The proof is similar to the one of Theorem 4.1 of \cite{BG21}. Nevertheless, here, we have a weaker control of the remainder term ($h^{-1/2}$ instead of $h^{1/2}$ in \cite{BG21}) and the vector field and Poisson bracket estimates of Section~\ref{sec:hamfo} generate new constants we have to track. As usual, we proceed by induction. More precisely, we choose $n\in \llbracket p,r+p\rrbracket$ as induction index and assume that Theorem \ref{thm_NF} holds if 
\begin{itemize}
\item we replace \eqref{eq:NF} by
\begin{equation}
\label{eq:newNF}
H \circ \tau^{(1)} = Z_2+ \sum_{j=p}^{r+p-1} Q^{(j)} + R \quad \mathrm{where} \quad Q^{(j)} \in \mathscr{H}_M^{j} \quad \mathrm{satisfies} \quad \|Q^{(j)}\|_{\mathscr{H}} \leq C B^\beta.
\end{equation}
\item we replace  \eqref{eq:ca_commte} by
\begin{equation}
\label{eq:ca_recommte}
\forall \ell \in \mathbb{N},\forall j\in \llbracket p,n-1\rrbracket, \ \langle \ell \rangle \leq N \ \Rightarrow \ \{ J_\ell , Q^{(j)} \} = 0.
\end{equation}
\item we replace \eqref{eq:dtau} by
\begin{equation}
\label{eq:redtau}
\|\mathrm{d} \tau^{(\nu)}(u) \|_{\mathscr{L}(h^{1/2})} \leq 2^{n-p} \quad \mathrm{and} \quad \|\mathrm{d} \tau^{(\nu)}(u) \|_{\mathscr{L}(h^{-1/2})} \leq 2^{n-p}.
\end{equation}
\end{itemize}
Even if we do not write it explicitely, we note that each polynomial $Q^{(j)}$ depends implicitely on $n$ as well as $R$, $\varepsilon_2$ and $\tau^{(\nu)}$. Moreover, we suppose that $R$ verifies the quantitative estimates of the theorem and that each $Q^{(j)}$ enjoys the same norm estimate as $P^{(p)}$ up to increasing the value of the constant $C_p$ (in a way that depends only on $(n,\mu,a)$) and up to increasing the value of $a$ and $\beta$ (in a way that depends only on $(n,a)$). If $n=p$, there is nothing to do: it is in fact enough to choose $\tau^{(0)}=\tau^{(1)} = \mathrm{id}_{\mathbb{C}^{\mathcal{T}_M}}$, $R=0$, $Q^{(p)}=P^{(p)}$, $Q^{(j)}=0$ for $j>p$ and $\beta=a$. For the sake of clarity, we will denote with a symbol $\sharp$ the objects we are going to introduce at the step $n+1$ (e.g. $\tau^{(0)}_\sharp, \beta_\sharp$...). Before entering the details of the proof, recall that one goes formally from step $n$ to $n+1$ by conjugating the normal form~\eqref{eq:newNF} by the time one map of the Hamiltonian flow of some well chosen function $\chi$. The function $\chi$ is chosen in such a way that the terms of $Q^{(n)}$ that do not commute with the expected super actions are cancelled out by solving a certain cohomological equation.

\medskip

\noindent \underline{$\star$ \emph{Decomposition of $Q^{(n)}$}.} We split the polynomial $Q^{(n)}$ as $Q = L + U$, the Hamiltonians
 $L,U \in  \mathscr{H}_{M}^{n}$ being defined by
$$
 L_{\mathbf{k}}^\sigma = \left\{\begin{array}{cll} (Q^{(n)})_{\mathbf{k}}^\sigma & \mathrm{if} & \kappa(\sigma,{\mathbf{k}}) \leq N \\
0 & \mathrm{otherwise}
\end{array} \right. \quad \mathrm{and} \quad  U_{\mathbf{k}}^\sigma = \left\{\begin{array}{cll} 0 & \mathrm{if} &  \kappa(\sigma,{\mathbf{k}}) \leq N \\
(Q^{(n)})_{\mathbf{k}}^\sigma & \mathrm{otherwise}
\end{array} \right.
$$
where $\kappa(\sigma,\mathbf{k})$ is defined in \eqref{eq:defkappa} and denotes the smallest effective index of the small divisor $\Omega(\sigma,\mathbf{k})$ defined in~\eqref{eq:defOmega}. Observe that, since these Hamiltonians are extracted from $Q^{(n)}$, they enjoy the same norm estimates.

\medskip

\noindent \underline{$\star$ \emph{$U$ commutes with the low super-actions}.} Indeed, a direct computation shows that if $\langle \ell  \rangle\leq N$, we have
\begin{align*}
\{ J_\ell , U \} &=2 \, i   \sum_{\sigma \in \{-1,1\}^{n}} \sum_{\mathbf{k} \in \mathcal{T}_M^n } (\sigma_1 \mathbbm{1}_{\omega_{k_1} = \omega_{(\ell,0)}} + \dots + \sigma_n  \mathbbm{1}_{\omega_{k_n} = \omega_{(\ell,0)}})  U_{\mathbf{k}}^{\sigma} u_{k_1}^{\sigma_1} \dots  u_{k_n}^{\sigma_n} 
\\
&=2 \, i  \sum_{\sigma \in \{-1,1\}^{n}} \sum_{\mathbf{k} \in \mathcal{T}_M^n } \big(  \sum_{j : \exists m, \ k_j=(\ell,m)} \sigma_j\big)  U_{\mathbf{k}}^{\sigma} u_{k_1}^{\sigma_1} \dots u_{k_n}^{\sigma_n} 
.
\end{align*}
However, since $ \langle \ell \rangle \leq N$, by definition of $U$ and $\kappa$ (see \eqref{eq:defkappa}), either $ \sum_{j : \exists m, \ k_j=(\ell,m)} \sigma_j$  vanishes or $U_{\mathbf{k}}^{\sigma}$  vanishes. Consequently $U$ and $J_\ell$ commute : $\{ J_\ell , U \}(u)=0$. We emphasize that the definition of $\kappa$ as the smallest effective index is crucial here. Without it, we would need some smoothness assumption on $u$ to control these commutators. Due to that, we will have much more terms to solve in the upcoming cohomological equation but we will be able to handle these extra factors thanks to the control the small divisors by $\kappa$ given by Proposition~\ref{prop:strong_nr}.

\medskip

\noindent \underline{$\star$ \emph{The cohomological equation.}} The mass $\mu$ has been fixed to make the frequencies strongly non-resonant (according to Proposition \ref{prop:strong_nr}). Therefore, there exist $\gamma \in (0,1)$ (depending only on $(n,\mu)$) and $\alpha>1$ (depending only on $n$) such that 
\begin{equation}
\label{eq:lowboundOm}
\kappa(\sigma,{\mathbf{k}}) \leq N
\quad \Rightarrow \quad \Omega(\sigma,\mathbf{k}) \geq \gamma N^{-\alpha}=:\delta.
\end{equation}
Therefore we set $\chi \in \mathscr{H}_{M}^{n}$ the Hamiltonian defined by
$$
\chi_{\mathbf{k}}^\sigma := \frac{L_{\mathbf{k}}^\sigma}{i \Omega(\sigma,\mathbf{k})} \quad \mathrm{if} \quad \kappa(\sigma,{\mathbf{k}}) \leq N
\quad \mathrm{and} \quad \chi_{\mathbf{k}}^\sigma = 0 \quad \mathrm{otherwise}.
$$
A direct computation shows that $\chi$ is a solution of the cohomological equation 
\begin{equation}
\label{eq:coho}
\{ \chi, Z_2\} + L =0.
\end{equation}
Let us now verify that we have a good control of the $\mathscr{C}$-norm of $\chi$. First, the bounds 
$$
\forall y \geq 0, \quad |\langle y  \rangle - y | \leq 1 \quad \mathrm{and} \quad | \sqrt{y(y+1) + \mu} - y | \leq \mu + 1
$$
and the decomposition
$$
\langle \sum_{j=1}^n \sigma_j \ell_j  \rangle  = \Big ( \langle \sum_{j=1}^n \sigma_j \ell_j  \rangle  - \sum_{j=1}^n \sigma_j \ell_j   \Big) +  \sum_{j=1}^n \sigma_j (\ell_j - \omega_{k_j} ) + \Omega(\sigma,\mathbf{k}),
$$
where $k_j = (\ell_j,m_j)$ for all $j\in \llbracket 1,n\rrbracket $, provide the estimate
$$
\langle \sigma_1 \ell_1 + \cdots +  \sigma_n \ell_n  \rangle \leq (n+1)(\mu+1) + |\Omega(\sigma,\mathbf{k})|.
$$
Therefore,  as a consequence of \eqref{eq:lowboundOm} (since $\delta < 1$) we have the bound
$$
|\chi_{\mathbf{k}}^\sigma| \leq (n+2)(\mu +1) \delta^{-1}  \frac{|L_{\mathbf{k}}^\sigma|}{\langle \sigma_1 \ell_1 + \cdots +  \sigma_n \ell_n  \rangle}
$$
and so
$$
\|\chi\|_{\mathscr{C}} \lesssim_{n,\mu} \delta^{-1}\| L \|_{\mathscr{H}} \lesssim_{n,\mu} \delta^{-1} \| Q^{(n)} \|_{\mathscr{H}} \lesssim_{n,\mu} \delta^{-1} C B^\beta.
$$

\medskip

\noindent \underline{$\star$ \emph{The new variables}.} As usual, we have to compose the change of variables $\tau$ at step $n$ with the Hamiltonian flow of $\chi$ (see \eqref{eq:mozart} below). Since they are only defined locally, we have to pay attention to their domains of definition. Eventhough the overall strategy is clear, it is a little bit tedious to check.

Since $\|\chi\|_{\mathscr{C}}  \lesssim_{n} \delta^{-1} C B^\beta$ and $\gamma N^{-\alpha}=:\delta$, applying Proposition \ref{prop_ham_flow}, we get a constant $K>0$ depending only on $(n,C, \mu)$, an exponent $b>0$ depending only on $(n,\beta)$ such that setting $\varepsilon_1 = (K  B^{b})^{-1/(n-2)}$, $\chi$ generates a smooth map 
$$
\Phi_\chi : \left\{ \begin{array}{cll} [-1,1] \times B_{h^{1/2}(\mathbb{C}^{\mathcal{T}_M})}(0,\varepsilon_1) &\to& \mathbb{C}^{\mathcal{T}_M} \\ (t,u) &\mapsto& \Phi_\chi^t(u) \end{array} \right.
$$
solving the equation 
$
-i\partial_t \Phi_\chi = ( \nabla \chi)\circ \Phi_\chi,$
and such that for all $t\in [-1,1]$, $\Phi_\chi^t$ is symplectic, close to the identity
\begin{equation}
\label{eq_dindon_new}
\| u \|_{h^{1/2}} < \varepsilon_1  \quad \Rightarrow \quad \| \Phi_\chi^t u - u \|_{h^{1/2}} \leq  \left( \frac{\| u \|_{h^{1/2}}}{\varepsilon_1}\right)^{n-2} \| u \|_{h^{1/2}},
\end{equation}
invertible
\begin{equation}
\label{eq_invertibility_new}
\|\Phi_\chi^{-t} (u) \|_{h^{1/2}} < \varepsilon_1 \quad \Rightarrow \quad  \Phi_\chi^{t}\circ  \Phi_\chi^{-t} (u) = u.
\end{equation}
Moreover, the map $u \mapsto \mathrm{d} \Phi_\chi^t(u) $ is continuous and we have the estimates
\begin{equation}
\label{eq_petanque_new}
\| u \|_{h^{1/2}} < \varepsilon_1  \quad \Rightarrow \quad  \| \mathrm{d} \Phi_\chi^t (u)\|_{\mathscr{L}(h^{1/2})}\leq 2 \quad \mathrm{and} \quad  \| \mathrm{d} \Phi_\chi^t (u)\|_{\mathscr{L}(h^{-1/2})}\leq 2.
\end{equation}
As usual, we aim at defining, for a proper choice of $\varepsilon_2^\sharp$,
\begin{equation}
\label{eq:mozart}
\tau^{(1)}_\sharp := \tau^{(1)} \circ \Phi_\chi^1\quad \mathrm{on} \quad B_{h^{1/2}}(0,2 \varepsilon_2^\sharp)\quad \mathrm{and} \quad \tau^{(0)}_\sharp := \Phi_\chi^{-1}\circ \tau^{(0)}\quad \mathrm{on} \quad B_{h^{1/2}}(0, \varepsilon_2^\sharp).
\end{equation}
To ensure that such a definition makes sense, we have to choose $\varepsilon_2^\sharp$ in such a way that
\begin{equation}
\label{eq:relou1}
2\varepsilon_2^\sharp \leq \varepsilon_1 \quad \mathrm{and} \quad (\| u\|_{h^{1/2}} < 2\varepsilon_2^\sharp \quad \Rightarrow \quad  \| \Phi_\chi^1(u) \|_{h^{1/2}} < 2 \varepsilon_2).
\end{equation}
\begin{equation}
\label{eq:relou2}
\varepsilon_2^\sharp \leq \varepsilon_2  \quad \mathrm{and} \quad (\| u\|_{h^{1/2}} < \varepsilon_2^\sharp \quad \Rightarrow \quad  \| \tau^{(0)}(u) \|_{h^{1/2}} <  \varepsilon_1).
\end{equation}
Let us analyze these conditions. First, we focus on \eqref{eq:relou1}. Provided that $\| u\|_{h^{1/2}} < 2\varepsilon_2^\sharp \leq \varepsilon_1$, since $ \Phi_\chi^1$ is close to the identity (see \eqref{eq_dindon_new}), we have $\Phi_\chi^1(u) \leq 2 \| u\|_{h^{1/2}} < 4 \varepsilon_2^\sharp$. Therefore, to get \eqref{eq:relou1} it is enough to have $2\varepsilon_2^\sharp \leq \min(\varepsilon_2,\varepsilon_1)$. Similarly, since $\tau^{(0)}$ is close to the identity (see \eqref{eq:taucloseid}), to get \eqref{eq:relou2} it is enough to ensure that $2\varepsilon_2^\sharp \leq \varepsilon_1$ and $\varepsilon_2^\sharp \leq \varepsilon_2$.

Before fixing $\varepsilon_2^\sharp$, let us only assume that $2\varepsilon_2^\sharp \leq \min(\varepsilon_2,\varepsilon_1)$ and investigate which conditions $\varepsilon_2^\sharp$ has to satisfy to ensure that $\tau^{(1)}_\sharp$ and $\tau^{(0)}_\sharp $ enjoy the properties described in Theorem \ref{thm_NF} (close to the identity, invertible...).

First, let us note that $\tau^{(1)}_\sharp$ and $\tau^{(0)}_\sharp$ are obviously symplectic and their differentials enjoy the bounds \eqref{eq:redtau} thanks to~\eqref{eq_petanque_new} (with $n \rightarrow n+1$). Hence, it remains to prove that $\tau^{(0)}_\sharp$ and $\tau^{(1)}_\sharp$ are close to the identity in the sense of~\eqref{eq:taucloseid}. To that aim, if $\| u \|_{h^{1/2}}< \varepsilon_2^\sharp$, since both  $\Phi_\chi^{-1}$ and $\tau^{(0)} $ are close to the identity, then we have
\begin{align*}
\|\tau^{(0)}_\sharp(u)-u\|_{h^{1/2}}  &\leq \left( \frac{\| \tau^{(0)}(u) \|_{h^{1/2}}}{\varepsilon_1}\right)^{n-2}  \|  \tau^{(0)}(u) \|_{h^{1/2}} + \left( \frac{\| u \|_{h^{1/2}}}{\varepsilon_2}\right)^{p-2} \|  u \|_{h^{1/2}} \\
&\leq \left( \frac{2\| u \|_{h^{1/2}}}{\varepsilon_1}\right)^{n-2} 2 \| u \|_{h^{1/2}} + \left( \frac{\| u \|_{h^{1/2}}}{\varepsilon_2}\right)^{p-2} \|  u \|_{h^{1/2}}.
\end{align*}
Therefore, since $n\geq p$ and $2\| u \|_{h^{1/2}}< 2\varepsilon_2^\sharp \leq \varepsilon_1 $, we deduce that
$$
\|\tau^{(0)}_\sharp(u)-u\|_{h^{1/2}} \leq \left( \frac{\| u \|_{h^{1/2}}}{\varepsilon_2^\sharp}\right)^{p-2} \|  u \|_{h^{1/2}}  \left[ \frac{2 (\varepsilon_2^\sharp)^{p-2}}{  \varepsilon_1^{p-2}} +  \frac{ (\varepsilon_2^\sharp)^{p-2} }{ \varepsilon_2^{p-2} } \right].
$$ 

As a consequence, since $p \geq 3$, if $3\varepsilon_2^\sharp \leq \min(\varepsilon_2,\varepsilon_1)$, we deduce that both $(\varepsilon_2^\sharp)^{p-2}/ \varepsilon_1^{p-2}$ and $(\varepsilon_2^\sharp)^{p-2}/ \varepsilon_2^{p-2}$ are bound by $1/3$ and so that $\tau^{(0)}_\sharp$ is close to the identity. It can be proven, with a similar decomposition, that if $6\varepsilon_2^\sharp \leq \min(\varepsilon_2,\varepsilon_1)$ then $\tau^{(1)}_\sharp$ is also close to the identity.

Finally, we also note that if $\tau^{(0)}_\sharp$ is close to the identity then it takes values in $B_{h^{1/2}}(0,2\varepsilon_2^\sharp )$.  Thus, as $\Phi_\chi^1$ is invertible (see \eqref{eq_invertibility_new}), the diagram \eqref{mon_beau_diagram_roi_des_forets} associated with $\tau^{(0)}_\sharp$ and $\tau^{(1)}_\sharp$ commutes.

To conclude this paragraph, we fix $\varepsilon_2^\sharp$ as large as possible to get all the properties of $\tau^{(0)}_\sharp$ and $\tau^{(1)}_\sharp$ , i.e.
 $$
 \varepsilon_2^\sharp = \frac16 \min(\varepsilon_2,\varepsilon_1).
 $$
We note that, therefore, we have $\varepsilon_2^\sharp \geq \frac16 \min( (K  B^{b})^{-1/(n-2)} , (C B^{\beta}))^{-1} ) \geq (C_\sharp B^{\beta_\sharp})^{-1}$ provided that $C_\sharp \geq 6 \max( K^{1/(n-2)},C)$ and $\beta_\sharp \geq \max(b/(n-2),\beta)$ (these constant will be determined at the end of the proof).

\medskip

\noindent \underline{$\star$ \emph{The new Hamiltonian.}} We aim at describing the Taylor expansion of $H\circ \tau^{(1)}_\sharp$. Since $t\mapsto \Phi_\chi^t$ is a smooth function solving the equation $-i\partial_t \Phi_\chi = ( \nabla \chi)\circ \Phi_\chi$, realizing a Taylor expansion in $t=0$ (on $B_{h^{1/2}}(0,2 \varepsilon_2^\sharp)$) gives
\begin{multline*}
H\circ \tau^{(1)}_\sharp = H\circ \tau^{(1)} \circ \Phi_\chi^1 
= Z_2 \circ  \Phi_\chi^1  + \sum_{j=p}^{r+p-1} Q^{(j)} \circ  \Phi_\chi^1    + R \circ \Phi_\chi^1 \\
= Z_2 + \sum_{j=p}^{r+p-1} Q^{(j)} + \{\chi,Z_2\}   + \sum_{h=1}^{m_{n}} \frac1{(h+1) !} \mathrm{ad}_{\chi}^{h+1} Z_2 + \sum_{j=p}^{r+p-1} \sum_{h=1}^{m_{j}} \frac1{h !} \mathrm{ad}_{\chi}^h Q^{(j)} + R \circ \Phi_\chi^1 \\
+  \int_0^1 \Big(\frac{(1-t)^{m_{n} +1 }}{(m_{n} +1) !}  (\mathrm{ad}_{\chi}^{m_{n} +2} Z_2)\circ \Phi_\chi^t  +   \sum_{j=p}^{r+p-1} \frac{(1-t)^{m_j}}{m_j !}  (\mathrm{ad}_{\chi}^{m_j+1} Q^{(j)}) \circ \Phi_\chi^t\Big)  \  \mathrm{d}t 
\end{multline*}
where $m_j$ denotes the largest integer such that $j+m_j (n -2) < r+p$ and $\mathrm{ad}_\chi := \{\chi,\cdot\}.$ 

In order to pool these terms by packets, we recall that by construction $\{\chi,Z_2\} = -L$ is of order $n$, that $\chi \in \mathscr{H}_{M}^{n}$ is of degree $n$ and that the Poisson bracket of two homogeneous polynomials of degree $r_1$ and $r_2$ is of degree $r_1+r_2-2$. Therefore we set
$$
Q^{(j)}_\sharp = Q^{(j)} \quad \mathrm{if} \quad j < n, \quad Q^{(n)}_\sharp = Q^{(n)}+  \{\chi,Z_2\} =  Q^{(n)} - L =U,
$$
$$
Q^{(j)}_\sharp =  \sum_{ j_\star+h(n - 2)= j} \frac1{h !} \mathrm{ad}_{\chi}^h Q^{(j_\star)} - \sum_{ n+h(n - 2)= j} \frac1{(h+1) !} \mathrm{ad}_{\chi}^h L  \quad \mathrm{if} \quad j > n,
$$
$$
R_\sharp=  R \circ \Phi_\chi^1 
-  \int_0^1 \Big(\frac{(1-t)^{m_{n} +1 }}{(m_{n} +1) !}  (\mathrm{ad}_{\chi}^{m_{n} +1} L)\circ \Phi_\chi^t  +   \sum_{j=p}^{r+p-1} \frac{(1-t)^{m_j}}{m_j !}  (\mathrm{ad}_{\chi}^{m_j+1} Q^{(j)}) \circ \Phi_\chi^t\Big)  \  \mathrm{d}t,
$$
where $h$ and $j_\star$ are the indices on which the sums hold in the definition of $Q^{(j)}_\sharp$.\\

 If  $j\leq n$, $Q^{(j)}_\sharp \in \mathscr{H}_{M}^{j}$ commutes with the low super-actions\footnote{Note that $U$ has been designed to get this property.} and we have
$$
\| Q^{(j)}_\sharp \|_{\mathscr{H}} \leq \| Q^{(j)} \|_{\mathscr{H}} \leq C B^{\beta}.
$$
If $j>n$, we have  $Q^{(j)}_\sharp \in \mathscr{H}_{M}^{j}$ and we apply Proposition \ref{poisson} to estimate its norm. Indeed if $j_\star+h(n - 2)= j$, we can use our estimate on $ \| \chi\|_{\mathscr{C}}$ to derive that
\begin{align*}
 \| \mathrm{ad}_{\chi}^h Q^{(j_\star)} \|_{\mathscr{H}} \lesssim_r (\log M)^{h} \| \chi\|_{\mathscr{C}}^{h} \| Q^{(j_\star)} \|_{\mathscr{H}} &\lesssim_r (\gamma^{-1} N^{\alpha}  \log M)^h (C B^\beta)^{h+1} \\ &\lesssim_r \gamma^{-h} C^{h+1} B^{h (\alpha+1) + (h+1) \beta}
\end{align*}
Similarly, $L$ enjoying the same bound as $Q^{(n)}$
, if 
$n+h(n - 2)= j$, we have, $ \| \mathrm{ad}_{\chi}^h L
\|_{\mathscr{H}} \lesssim_r  \gamma^{-h} C^{h+1} B^{h (\alpha+1) + (h+1) \beta}$. As a consequence, since $h\leq r+p$, provided that $C_\sharp \gtrsim_r  \gamma^{-r-p} C^{r+p+1} $ and $\beta_\sharp \geq (r+p) (\alpha+1) + (r+p+1) \beta$, we have 
$\| Q^{(j)}_\sharp \|_{\mathscr{H}} \leq C_\sharp B^{\beta_\sharp}$ for $j>n$.

\medskip

\noindent \underline{$\star$ \emph{Control of the remainder term.}} Now we are left with controlling $\nabla R_\sharp$ in $h^{-1/2}$. We fix $u \in \mathbb{C}^{\mathcal{T}_M}$ such that $\|u\|_{h^{1/2}} < 2 \varepsilon_2^\sharp$. First we focus on $R \circ \Phi_\chi^1(u)$. By composition, we have
$$
\nabla (R \circ \Phi_\chi^1)(u) = (\mathrm{d}\Phi_\chi^1(u))^* (\nabla R) \circ \Phi_\chi^1(u).
$$
where $ (\mathrm{d}\Phi_\chi^1(u))^* \in \mathscr{L}(\mathbb{C}^{\mathcal{T}_M})$ denotes the adjoint of $\mathrm{d}\Phi_\chi^1(u)$. Moreover, by duality, we have $\|(\mathrm{d}\Phi_\chi^1(u))^* \|_{\mathscr{L}(h^{1/2})} = \|\mathrm{d}\Phi_\chi^1(u) \|_{\mathscr{L}(h^{-1/2})} \leq 2$ . Therefore, since $\|\nabla R(u)\|_{h^{-1/2}} \leq C B^{\beta} \|u\|_{h^{1/2}}^{r+p-1}$ and $\| \Phi_\chi^1(u) \|_{h^{1/2}} \leq 2 \| u\|_{h^{1/2}}$, we have
$$
\| \nabla (R \circ \Phi_\chi^1)(u) \|_{h^{-1/2}} \leq 2^{r+p} C B^\beta \| u\|_{h^{1/2}}^{r+p-1}.
$$

Now, we focus on $(\mathrm{ad}_{\chi}^{m_j+1} Q^{(j)}) \circ \Phi_\chi^t(u) $ where $p\leq j \leq r+p-1$ and $t\in [0,1]$. Arguing as above and using Proposition \ref{poisson} to estimate the norm of the Poisson brackets and Proposition \ref{grad-} to estimate the norm of the gradient, we have
\begin{equation*}
\begin{split}
\| \nabla ((\mathrm{ad}_{\chi}^{m_j+1} Q^{(j)}) \circ \Phi_\chi^t)(u) \|_{h^{-1/2}} &\leq 2 \|  (\nabla (\mathrm{ad}_{\chi}^{m_j+1} Q^{(j)}))\circ \Phi_\chi^t (u)  \|_{h^{-1/2}} \\
&\lesssim_{r,\mu}  (\delta^{-1}  \log M)^{m_j+1} (C B^\beta)^{m_j+2} (\log M)^{r_j/2}\|  \Phi_\chi^t (u)  \|^{r_j-1}_{h^{1/2}}.
\end{split}
\end{equation*}
where $r_j = j+(m_j+1)(n - 2) \in \llbracket r+p,2(r+p)\rrbracket$ (by definition of $m_j$). Thus, provided that 
$$
C_\sharp \gtrsim_{r,\mu} \gamma^{-r-p-1} C^{r+p+2} \quad \mathrm{and} \quad \beta_\sharp \geq(\alpha+1) (r+p+1) + \beta (r+p+2) + r+p 
$$
we have $ \| \nabla ((\mathrm{ad}_{\chi}^{m_j+1} Q^{(j)}) \circ \Phi_\chi^t)(u) \|_{h^{-1/2}} \leq C_\sharp B^{\beta_\sharp} \|  u  \|^{r+p-1}_{h^{1/2}}$. As above, the argument works as well for the term involving $L$ as it enjoys the same norm estimate as $Q^{(n)}$.

Hence, if moreover, $\beta_\sharp \geq \beta$ and $C_\sharp \gtrsim_r C$ (to control $R \circ \Phi_\chi^1(u)$), we have
$$
\| \nabla R_\sharp(u) \|_{h^{-1/2}} \leq  C_\sharp B^{\beta_\sharp} \| u\|_{h^{1/2}}^{r+p-1}.
$$

\medskip

\noindent \underline{$\star$ \emph{Choice of $C_\sharp$ and $\beta_\sharp$.}} To conclude our induction step (and thus the proof), we just have to pick the smallest constants enjoying all the constraints (and to note that they do not depend on $B$)
$$
\beta_\sharp = (\alpha+1) (r+p+1) + \beta (r+p+2) + r+p \quad \mathrm{and} \quad C_\sharp \simeq_r  \max(\gamma^{-r+p-1} C^{r+p+2},K^{1/(n-2)}).
$$ 

\end{proof}

\section{Proofs of the main results }
\label{sec-proofs}
This final section is devoted to the proof of Theorem \ref{thm:main} and its Corollary \ref{cor:main}.

\subsection{On the global well-posedness of \eqref{eq:KG}}
\label{sub:LWP}
 In dimension $2$, the Sobolev norm $H^1$ controls all the Lebesgue norms $L^q$, $2\leq q<\infty$. Therefore, a standard fixed point argument (which does not require any kind of Strichartz estimate) provides the local well-posedness of the nonlinear 
 Klein-Gordon equation~\eqref{eq:KG} on the sphere $\mathbb{S}^2$ in the energy space $H^1 \times L^2$ (see e.g. Thm 6.2.2 page 83 of \cite{CH98}).

This nonlinear equation is Hamiltonian because it writes formally 
\begin{equation}
\label{eq:canham}
\partial_t \begin{pmatrix} \Phi \\ \partial_t \Phi  \end{pmatrix} = \begin{pmatrix} 0& 1  \\
-1 & 0 \end{pmatrix} \nabla \mathcal{H} (\Phi ,\partial_t \Phi )
\end{equation}
where the Hamiltonian $\mathcal{H}$ is given by \eqref{eq:hamKg}. Therefore, $\mathcal{H}$ is a constant of the motion of \eqref{eq:KG} (see e.g. Prop 6.2.3 page 83 of \cite{CH98}). It is especially useful since, as stated in the following lemma, it is uniformly elliptic in a neighborhood of the origin:
\begin{lemma}
\label{lem:ell}
For all $g\in L^\infty(\mathbb{S}^2;\mathbb{R})$ and all $\mu >0$, there exist $C>1$ and $\varepsilon_0>0$ such that for all $(\Phi,\Psi) \in H^1 \times L^2(\mathbb{S}^2;\mathbb{R})$, provided that $\| \Phi \|_{H^1} + \|\Psi\|_{L^2} \leq \varepsilon_0$, we have
$$
C^{-1}(\| \Phi \|_{H^1} + \|\Psi\|_{L^2} )^2 \leq \mathcal{H}(\Phi,\Psi) \leq C (\| \Phi \|_{H^1} + \|\Psi\|_{L^2} )^2.
$$ 
\end{lemma}
\begin{proof} It follows directly from the Sobolev embedding $H^1 \hookrightarrow L^p$ and from the fact that $p\geq 3$.
\end{proof}
As a consequence, as stated in the following proposition we get the global well-posedness of \eqref{eq:KG} in a neighborhood of the origin in $H^1 \times L^2$ (see e.g. Prop 6.3.3 page 84 of \cite{CH98}).
\begin{proposition} \label{prop:GWP} For all $\mu>0$ and all $g\in L^\infty$, there exist $\varepsilon_1>0$ and $K>1$ such that, as soon as $\varepsilon := \| \Phi^{(0)} \|_{H^1} + \|\dot{\Phi}^{(0)}\|_{L^2} \leq \varepsilon_1$, there exists a unique $\Phi \in C^0(\mathbb{R};H^1) \cap C^1(\mathbb{R};L^2) \cap C^2(\mathbb{R};H^{-1})$ solution to \eqref{eq:KG}.
Moreover, it enjoys the bound
$$
\forall t\in \mathbb{R}, \quad \ \| \Phi(t) \|_{H^1} + \| \partial_t \Phi(t) \|_{L^2} \leq K \varepsilon.
$$
\end{proposition}

\subsection{Proof of Theorem \ref{thm:main}} 

One more time, we fix the mass $\mu>0$ (in a set of full measure) to make the frequencies ($\omega_{(\ell,m)} = \sqrt{\ell (\ell+1) + \mu}$)  non-resonant in the sense of Proposition \ref{prop:strong_nr}. The strategy is the following. Using the above a priori estimates, we prove that the high super actions are under control as long as $N=\langle\ell\rangle\gtrsim\varepsilon^{-\frac{p-2}{\alpha_r+1}}$ for an arbitrary $\alpha_r>1$. Thus, we only have to deal with the low super-actions that we handle using the Birkhoff normal form of Theorem~\ref{thm_NF}. This requires to make a truncation of the frequency up to a certain level $M$ in order to reduce to the finite dimensional situation of this theorem. In order to ensure that all the remainder terms are small in this reduction to finite dimension, we need to take $M$ of order $\varepsilon^{-r}$. Then the conclusion follows by combining our a priori estimates on the solution with the normal form of Theorem~\ref{thm_NF} and by taking $\alpha_r$ larger than the exponent $\beta$ appearing in the remainder terms of that statement.

\medskip

\noindent \underline{ \emph{$\triangleright$ \eqref{eq:KG} as a Schr\"odinger equation.}} We consider $(\Phi^{(0)} ,\dot{\Phi}^{(0)})\in H^1 \times L^2$, satisfying $\varepsilon := \| \Phi^{(0)} \|_{H^1} + \|\dot{\Phi}^{(0)}\|_{L^2} < \varepsilon_0 \leq \varepsilon_1$ where $\varepsilon_0$ will be determined at the end of the proof and $\varepsilon_1$ is given by Proposition \ref{prop:GWP}. Thanks to this proposition, one obtains a global solution $\Phi$ to \eqref{eq:KG}. Then, in order to diagonalize the linear part of \eqref{eq:KG}, we set (as usual)
$$
u := \Lambda \Phi + i \Lambda^{-1} \partial_t\Phi \quad \mathrm{where} \quad \Lambda:=(  \mu - \Delta)^{1/4}.
$$
Indeed, $u$ belongs to $C^0(\mathbb{R};H^{1/2}) \cap C^1(\mathbb{R};H^{-1/2}) $ and solves the equation
\begin{equation}
\label{eq:schro}
i\partial_t u = \Lambda^2 u - \Lambda^{-1} \big(g \,  [ \Lambda^{-1} \Re u ]^{p-1}\big).
\end{equation}
It is relevant to note that the harmonic energies $\mathcal{E}_\ell$ (defined by \eqref{eq:def_harm_en}), that we aim at controlling in Theorem \ref{thm:main}, satisfy
$$
\forall \ell \in \mathbb{N}, \quad \mathcal{E}_{\ell} (\Phi(t)) = \| \Pi_\ell u(t) \|_{L^2}^2 := J_\ell(u(t))
$$
where $\Pi_\ell$ is the orthogonal projection on the eigenspace $E_\ell$ as defined in~\eqref{eq:lapdiag}.
Moreover, as a consequence of Proposition \ref{prop:GWP}, there exists a constant $K'>1$ depending only on $\mu$ such that
\begin{equation}
\label{eq:jecontrolelanorme}
\forall t\in \mathbb{R},  \quad \ \|u(t) \|_{H^{1/2}} \leq K' \varepsilon.
\end{equation}

\medskip

\noindent \underline{ \emph{$\triangleright$ The $N$-truncation.}} 
The control of the high super-actions is a direct consequence of the a priori bound \eqref{eq:jecontrolelanorme}. Indeed, applying the triangular inequality, we have
$$
|J_\ell(u(t)) - J_\ell(u(0)) | \leq J_\ell(u(t))   + J_\ell(u(0)) \leq 2 \langle \ell \rangle^{-1} \| u \|_{L^\infty_t H^{1/2}_x}^2 \leq 2 \langle \ell \rangle^{-1}  (K')^2 \varepsilon^2. 
$$
Being given $\alpha_r>1$ (depending only on $r$) that will be optimized at the end of the proof, we set
$$
 N^{(\max)}:= \varepsilon^{- \frac{p-2}{\alpha_r+1}} .
$$
As a consequence, for all $t\in \mathbb{R}$, we have (for all $t\in \mathbb{R}$)
\begin{equation}
\label{eq:conthighmodes}
\langle \ell \rangle \geq N^{(\max)} \quad\Rightarrow \quad |J_\ell(u(t)) - J_\ell(u(0)) | \lesssim_{r,\mu} \langle \ell \rangle^{\alpha_r} \varepsilon^p.
\end{equation}
Hence, from now on, we will only focus on the variations of the low super-actions. More precisely, we fix $\ell_\star \in \mathbb{N}$ and $N\in \mathbb{R}$ such that 
$$
N:= \langle \ell_\star  \rangle < N^{(\max)}
$$
and we aim at estimating the variations of $J_{\ell_{\star}}( u)$.


\medskip

\noindent \underline{ \emph{$\triangleright$ The $M$-truncation.}} In order to reduce ourselves to the finite dimensional situation of our Birkhoff normal form Theorem~\ref{thm_NF}, we are going to prove that the high enough modes (larger than $M\gg 1$) do not play any role in the dynamics for very long times (in $H^{-1/2}$). Let $M \geq 2N^{(\max)}$ be a constant that will be optimized later with respect to $\varepsilon$ and $\Pi_{\leq M}$ be the orthogonal projection on $\bigoplus_{ \ell \leq M } E_\ell$, i.e.
$$
\Pi_{\leq M} := \sum_{ \ell  \leq M} \Pi_\ell \quad \mathrm{and} \quad \Pi_{> M} := \mathrm{Id}_{L^2} - \Pi_{\leq M}.
$$
We set
$$
F^{(>M)}(t) := \Pi_{\leq M} [\mathcal{ N}(\Pi_{\leq M} u(t)) - \mathcal{ N}(u(t)) ] \quad \mathrm{where} \quad \mathcal{N}(u) := \Lambda^{-1} \big(g \,  [ \Lambda^{-1} \Re u ]^{p-1}\big).
$$
Since $u$ solves the equation \eqref{eq:schro}, $u^{(\leq M)} := \Pi_{\leq M} u(t)$ solves the \emph{non-autonomous} equation
\begin{equation}
\label{eq:nono}
i \partial_t u^{(\leq M)} = \Lambda^2 u^{(\leq M)}  -  \Pi_{\leq M}\mathcal{ N}(u^{(\leq M)}) + F^{(>M)}(t).
\end{equation}
We note that, since $ M \geq 2N^{(\max)} $, we have $M > \ell_\star$ and so 
\begin{equation}\label{eq:truncate-action}
 J_{\ell_\star}(u^{(\leq M)}) =  J_{\ell_\star} ( u).
\end{equation}
We aim at proving that the non-autonomous part of \eqref{eq:nono} (i.e. $F^{(>M)}(t)$) is negligible provided that $M$ is large enough. Indeed, as a consequence of the Sobolev embeddings $H^1 \hookrightarrow L^{6(p-2)}  \hookrightarrow L^{3/2} \hookrightarrow H^{-1}$, by H\"older and the mean value inequality, we have (uniformly with respect to $t$)

\begin{equation*}
\begin{split}
\| F^{(>M)} \|_{H^{-1/2}} &\lesssim_\mu \|  g\Phi^{p-1} -  g(\Pi_{\leq M} \Phi)^{p-1}  \|_{H^{-1}} \\
&\lesssim_{\mu, g} \|  \Phi^{p-1} - (\Pi_{\leq M} \Phi)^{p-1}  \|_{L^{3/2}} \\
&\lesssim_{\mu, g} \| (\Pi_{>M} \Phi ) (| \Pi_{\leq M} \Phi  |^{p-2} +|  \Phi  |^{p-2}   )  \|_{L^{3/2}} \\
&\lesssim_{\mu, g} \| \Pi_{>M} \Phi \|_{L^2} (\|(  \Pi_{\leq M} \Phi  )^{p-2} \|_{L^6} + \|   \Phi  ^{p-2} \|_{L^6})  \\
&\lesssim_{\mu, g} M^{-1} \|  \Phi \|_{H^1}^{p-1} \lesssim_{\mu,g} M^{-1} \varepsilon^{p-1}.
\end{split}
\end{equation*}
Therefore, from now, we assume that $M \geq \varepsilon^{- r},$ and we get
$$
\forall t\in \mathbb{R}, \quad \| F^{(>M)} (t)\|_{H^{-1/2}} \lesssim_{\mu} \varepsilon^{r+p-1}.
$$

\medskip

\noindent \underline{ \emph{$\triangleright$ Discretization.}} Thanks to Theorem \ref{t:proba}, we get a basis $(e_k)_{k\in \mathcal{T}_\infty}$ of $L^2$ which diagonalizes the Laplace--Beltrami operator $\Delta$ and enjoys nice algebraic properties. In particular, thanks to this basis, we identify $\bigoplus_{ \ell \leq M } E_\ell$ with $\mathbb{R}^{\mathcal{T}_M}$ (and the usual Sobolev norms with the discrete ones).

We use this basis to rewrite the autonomous part of \eqref{eq:nono} as a Hamiltonian system :
\begin{equation}
\label{eq:hamfor}
i \partial_t u^{(\leq M)} = \nabla H(u^{(\leq M)} )+ F^{(>M)}(t).
\end{equation}
where
$$
H = Z_2 + P^{(p)} \quad \mathrm{with} \quad Z_2(u) = \frac12 \sum_{k\in \mathcal{T}_M} \omega_k |u_k|^2
$$
and $P^{(p)} \in \mathscr{H}_M^p$ is defined, for all $\mathbf{k}=(k_1,\ldots,k_p) \in \mathcal{T}_M^p$ and $\sigma \in \{-1,1\}^p$ by
$$
(P^{(p)})_{\mathbf{k}}^\sigma = - \frac1{p 2^p} \left(\prod_{j=1}^p \frac1{(\ell_j (\ell_j +1)+\mu)^{1/4}} \right) \int_{\mathbb{S}^2} e_{k_1}(x) \cdots e_{k_p}(x) g(x)  \mathrm{d}\mathrm{vol}_{\mathbb{S}^2}(x).
$$
Thanks to Theorem \ref{t:proba}, the basis $(e_k)_{k\in \mathcal{T}_\infty}$ has been chosen such that
\begin{equation}
\label{eq:toutcapourca}
\| P^{(p)} \|_{\mathscr{H}} \lesssim (\log(M))^p.
\end{equation}
Note that the choice of the orthonormal basis of Theorem \ref{t:proba} is crucial here. With the standard basis of spherical harmonics we would not get such a good control on the nonlinearity.

\medskip

\noindent \underline{ \emph{$\triangleright$ Change of variables.}} Now, we apply Theorem \ref{thm_NF} (i.e. our Birkhoff normal form result) to simplify the Hamiltonian part of \eqref{eq:hamfor}. More precisely, we get some transformations $\tau^{(0)},\tau^{(1)}$, some Hamiltonians $Q_{res}^{\leq N}$ and $R$, some constants $C,\beta$ and $\varepsilon_2$ such that the statement of Theorem \ref{thm_NF} holds. We recall that $B$ is defined by $B= \max(N,\log(M))$.

We will optimize the constants in such a way that we have
$$
 K' \varepsilon < (C B^\beta)^{-1}.
$$
As a consequence, we have
$$
\forall t \in \mathbb{R}, \quad  \| u^{(\leq M)}(t) \|_{h^{1/2}} \leq  K' \varepsilon < (C B^\beta)^{-1} \leq \varepsilon_2.
$$
Therefore, it makes sense to define
$$
v := \tau^{(0)} \circ u^{(\leq M)}.
$$
Moreover since the diagram \eqref{mon_beau_diagram_roi_des_forets} commutes we have
$$
u^{(\leq M)} = \tau^{(1)}\circ v .
$$
As a consequence, since $\tau^{(0)}$ is symplectic and $(\mathrm{d}\tau^{(0)}(u^{(\leq M)}))^{-1} = \mathrm{d}\tau^{(1)}(v)$, we have
\begin{equation}
\label{eq:lader}
i\partial_t v(t) = \nabla (Z_2 + Q_{res}^{\leq N} )(v(t)) + W(t)
\end{equation}
where $W$ is the new remainder term defined by
$$
W(t) := \nabla R (v(t)) + \mathrm{d}\tau^{(0)}(u^{(\leq M)}(t)) (F^{(>M)}(t)).
$$
Let us estimate $W$. On the one hand, since $\tau^{(0)}$ is close to the identity in the sense of Theorem~\ref{thm_NF}, we have  
\begin{equation}
\label{eq:uvpetits} 
\|v(t)\|_{h^{1/2}}\leq \|u^{(\leq M)}(t)\|_{h^{1/2}} +\|v(t)-u^{(\leq M)}(t)\|_{h^{1/2}}  \leq 2\|u^{(\leq M)}(t)\|_{h^{1/2}}  \leq 2 K' \varepsilon\lesssim_{\mu} \varepsilon.
\end{equation}
Hence, thanks to Theorem~\ref{thm_NF}, we get $\| \nabla R (v(t)) \|_{h^{-1/2}} \lesssim_{r,\mu} B^\beta \varepsilon^{r+p-1}$. On the other hand, since $\mathrm{d}\tau^{(0)}(u^{(\leq M)}(t))$ is controlled in $\mathscr{L}(h^{-1/2})$ (by $2^r$), we deduce that 
$$\| \mathrm{d}\tau^{(0)}(u^{(\leq M)}(t)) (F^{(>M)}(t))\|_{h^{-1/2}} \lesssim_{r,\mu} \varepsilon^{r+p-1}.$$ Therefore, we have
\begin{equation}
\label{eq:estW}
\| W(t) \|_{h^{-1/2}} \lesssim_{r,\mu} B^\beta \varepsilon^{r+p-1}.
\end{equation}
Finally, let us note that, since $\tau^{(0)}$ is close to the identity in the sense of Theorem~\ref{thm_NF} and $(C B^\beta)^{-1} \leq \varepsilon_2$, we have
\begin{equation}
\label{eq:uaimev}
\|u^{(\leq M)}(t) -v(t) \|_{h^{1/2}} \lesssim_{r,\mu} \varepsilon^{p-1} B^{\beta(p-2)}.
\end{equation}

\medskip

\noindent \underline{ \emph{$\triangleright$ Control of the low super-actions.}}  As a consequence of~\eqref{eq:truncate-action}, \eqref{eq:uaimev} and \eqref{eq:uvpetits}, we have
$$
|J_{\ell_\star } ( u(t)) - J_{\ell_\star } ( v(t)) | \leq \|  u^{(\leq M)}(t) - v(t)   \|_{\ell^2} (\|  u^{(\leq M)}(t) \|_{\ell^2} + \| v(t)   \|_{\ell^2}) \lesssim_{r,\mu}  \varepsilon^{p} B^{\beta(p-2)}.
$$
Hence, by the triangular inequality, we have
$$
|J_{\ell_\star } ( u(t)) - J_{\ell_\star } ( u(0)) | \lesssim_r   |J_{\ell_\star } ( v(t)) - J_{\ell_\star } ( v(0)) | + \varepsilon^{p} B^{\beta(p-2)} . 
$$
However, since $v$ solves \eqref{eq:lader}, we have
$$
\partial_t J_{\ell_\star } ( v(t)) = \{  J_{\ell_\star } , Z_2 + Q_{res}^{\leq N} \}(v(t)) + ( i\nabla J_{\ell_\star }(v(t)) ,W(t)  )_{\ell^2}.
$$
By construction, since $\langle \ell_\star \rangle = N$,  $Z_2 + Q_{res}^{\leq N}$ and $J_{\ell_\star }$ commute, i.e. $\{  J_{\ell_\star } , Z_2 + Q_{res}^{\leq N} \} = 0.$
 As a consequence, using the estimate \eqref{eq:estW} on $W$, we have
\begin{equation*}
\begin{split}
|\partial_t J_{\ell_\star } ( v(t))  | \leq |( i\nabla J_{\ell_\star }(v(t)) ,W(t)  )_{\ell^2}|\leq \| \nabla J_{\ell_\star }(v(t)) \|_{h^{1/2}} \|W(t)\|_{h^{-1/2}} &\leq 2\| v(t) \|_{h^{1/2}}  \|W(t)\|_{h^{-1/2}} \\
&\lesssim_{r,\mu} B^\beta \varepsilon^{r+p}.
\end{split}
\end{equation*}
Consequently, while $|t|\leq \varepsilon^{-r}$, we have
\begin{equation}
\label{eq:contlowmodes}
|J_{\ell_\star } ( u(t)) - J_{\ell_\star }( u(0)) | \lesssim_{r,\mu} \varepsilon^{p} B^{\beta(p-2)} \lesssim_{r,\mu,\nu} \langle \ell_\star \rangle ^{\alpha_r} \varepsilon^{ p -\nu}
\end{equation}
provided that $B^{\beta(p-2)} \lesssim_{r,\mu,\nu}   N^{\alpha_r} \varepsilon^{-\nu}$ where $\nu>0$.

\medskip

\noindent \underline{ \emph{$\triangleright$ Conclusion.}} As we wanted, in \eqref{eq:conthighmodes} and \eqref{eq:contlowmodes}, we have controlled the variations of the super-actions. Nevertheless, to get these results we have done some assumptions on our parameters. Hence, to conclude, we have to check their compatibility and optimize them. 

More precisely, we have to prove that their exists $\alpha_r>1$ and $\varepsilon_0 \leq \varepsilon_1$ such that for all $\varepsilon < \varepsilon_0$ and all $N < N^{(\max)} =  \varepsilon^{- \frac{p-2}{\alpha_r+1}}$, there exists $M \geq 2$ satisfying
$$
\begin{array}{lcllll}
& (i) & B^{\beta(p-2)}\lesssim_{r,\mu,\nu}  N^{\alpha_r}\varepsilon^{-\nu} \quad  \quad & (ii) & K' \varepsilon < (C B^\beta)^{-1} \\
& (iii) &  M \geq  \varepsilon^{- r} & (iv) &   M \geq 2 N^{(\max)}
\end{array}
$$
where $B= \max(N,\log(M))$. First, we set $M =   \varepsilon^{- r} $ (so $(iii)$ is satisfied).
Then, we set $\alpha_r = \beta(p-2)$ and we note that the estimate $(i)$ holds. Finally, since $p\leq r$, we note that $(ii)$ and $(iv)$ are clearly satisfied provided that $\varepsilon_0$ is small enough.

\subsection{Proof of Corollary \ref{cor:main}} 
For all $t\in \mathbb{R}$, let $w(t) \in H^{1/2}(\mathbb{S}^2;\mathbb{C})$ be defined, for all $\ell \in \mathbb{N}$, by
$$
\Pi_\ell w(t) = \sqrt{ \frac{ J_\ell(u(0))}{ J_\ell(u(t))} } \, \Pi_\ell u(t)   \quad \mathrm{if} \quad J_\ell(u(t)) \neq 0 \quad  \mathrm{and} \quad \Pi_\ell w(t) = \Pi_\ell u(0) \quad \mathrm{else}.
$$
Indeed, recalling that $J_\ell = \|\Pi_\ell \cdot \|_{L^2}^2$, this function satisfies $\| w(t) \|_{H^{1/2}} = \| u(0) \|_{H^{1/2}} $ and
$$
\forall \ell \in \mathbb{N}, \quad J_\ell (w(t)) =J_\ell (u(0)) \quad \mathrm{and} \quad \sqrt{J_\ell(w(t) - u(t))} = |\sqrt{J_\ell(u(t))} - \sqrt{J_\ell(w(t))}|.
$$
As a consequence, applying Theorem \ref{thm:main} (with $\nu = 1/2$), while $|t|<\varepsilon^{-r}$, for all $\ell \in \mathbb{N}$, we have
$$
J_\ell(u(t) -w(t)) \leq |J_\ell(u(t)) -J_\ell(w(t))| = |J_\ell(u(t)) -J_\ell(u(0))| \lesssim_{\mu,r} \langle \ell \rangle^{\alpha_r} \varepsilon^{p-1/2}.
$$
Therefore, we have
$$
\| u(t) - w(t) \|_{H^{-\alpha_r/2}} \lesssim_{\mu,r}  \varepsilon^{(2p-1)/4}.
$$
Consequently, since $s<1/2$, setting $\theta = \max( 1, \frac{1-2s}{1+ \alpha_r}) $, by interpolation and using Proposition~\ref{prop:GWP}, we get
$$
\| u(t) - w(t) \|_{H^{s}} \lesssim_{r,s} \| u(t) - w(t) \|_{H^{1/2}}^{1-\theta} \| u(t) - w(t) \|_{H^{-\alpha_r/2}}^{\theta} \lesssim_{r,s,\mu} \varepsilon^{ 1 + \delta }
$$
where $\delta :=   \theta ((2p-1)/4 - 1) >0$ (because $p\geq 3$). 
Finally, to see that there exist some Hermitian operators $H_\ell(t) : E_\ell \otimes \mathbb{C} \to E_\ell \otimes \mathbb{C}$ such that
$$
\forall \ell \in \mathbb{N}, \quad \Pi_\ell w(t) = e^{i H_\ell(t)} \Pi_\ell u(0). 
$$
it is enough to note that the unitary group of $E_\ell \otimes \mathbb{C}$ acts transitively on the spheres  and that every unitary transform is the exponential of a skew-Hermitian operators (indeed, since $J_\ell (w(t)) =J_\ell (u(0))$, $\Pi_\ell w(t)$  and $\Pi_\ell u(0)$ belong to a same sphere).

\end{document}